\documentclass[11pt]{article}

\usepackage[margin=1.05in]{geometry}
\usepackage{amsmath,amssymb,amsthm,mathtools}
\usepackage{enumitem}
\usepackage[hidelinks]{hyperref}

\allowdisplaybreaks
\numberwithin{equation}{section}

\title{Second-Order Cover Time of Random Regular Graphs}
\author{%
  Lisa Hartung \qquad Andreas Klippel\\[0.5em]
  {\small Institut f\"ur Mathematik}\\
  {\small Johannes Gutenberg-Universit\"at Mainz}\\[0.3em]
  {\small \href{mailto:lhartung@uni-mainz.de}{\texttt{lhartung@uni-mainz.de}}\qquad
  \href{mailto:anklippe@uni-mainz.de}{\texttt{anklippe@uni-mainz.de}}}%
}
\date{\today}

\newtheorem{theorem}{Theorem}[section]
\newtheorem{proposition}[theorem]{Proposition}
\newtheorem{lemma}[theorem]{Lemma}
\newtheorem{corollary}[theorem]{Corollary}
\newtheorem{assumption}[theorem]{Assumption}
\theoremstyle{definition}

\theoremstyle{remark}
\newtheorem{remark}[theorem]{Remark}

\newcommand{\bbP}{\mathbb P}
\newcommand{\bbE}{\mathbb E}
\newcommand{\one}{\mathbf 1}
\newcommand{\Td}{\mathbb T_d}
\renewcommand{\Cap}{\operatorname{Cap}}
\newcommand{\capGT}{\operatorname{Cap}_{G,T}}

\newcommand{\dist}{\operatorname{dist}}
\newcommand{\TV}{\operatorname{TV}}
\newcommand{\diag}{\operatorname{diag}}
\newcommand{\op}{\operatorname{op}}

\begin{document}
\maketitle

\begin{abstract}
Let $G_N$ be a uniform random $d$-regular graph on $N$ vertices, with $d\ge3$
fixed.  Cooper and Frieze showed that the cover time of simple random walk on
$G_N$ is asymptotically $\theta_dN\log N$, where $\theta_d=(d-1)/(d-2)$.  We
prove that, in probability over the graph and uniformly in the
starting point, $(T_{\rm cov}-\theta_dN\log N)/(\theta_dN)$ converges in
distribution to a standard Gumbel law, confirming a prediction of Tishby,
Biham, and Katzav \cite{TishbyBihamKatzav2022}.  The proof reduces late points to avoidance probabilities
of finite sets via factorial moments, estimates these by Green-matrix
capacities, and compares the latter with capacities on the infinite
$d$-regular tree.  We also obtain Gumbel limits for lazy,
continuous-time and finitely many independent walks.  For the unvisited set,
we prove weighted laws of large numbers, spatial Poisson limits in the
cover-time window, and joint Poisson limits for fixed tree-component shapes
in their disappearance windows.
\end{abstract}

\section{Introduction}
\label{sec:introduction}

The cover time of a finite connected graph $G=(V,E)$ is the first time at
which simple random walk on $G$ has visited every vertex; see the monographs
\cite{AldousFill,LevinPeres2017} for background.  Its expectation lies between
$(1-o(1))n\log n$ and $(\frac4{27}+o(1))n^3$ for connected $n$-vertex
graphs~\cite{Feige1995lower,Feige1995upper}, it is concentrated whenever it
is much larger than the maximal expected hitting
time~\cite{Aldous1991threshold}, and it is determined up to constants by the
maximum of the associated Gaussian free field~\cite{DingLeePeres2012}.

Much less is known beyond first order.  For the complete graph the cover time
is a coupon-collector time with Gumbel fluctuations~\cite{ErdosRenyi1961}, and
Gumbel fluctuations were also proved for the discrete torus in dimension
$d\ge3$~\cite{Belius2013} and for bounded-degree vertex-transitive graphs in
a suitable regime~\cite{BerestyckiHermonTeyssier2022}.  In log-correlated settings the behaviour is
different: in two dimensions the first- and second-order terms are
known~\cite{DPRZ2004,BeliusKistler2017,Abe2021}, together with tightness on
the sphere~\cite{BeliusRosenZeitouni2020}, and for the binary tree the limit
law is a randomly shifted Gumbel
distribution~\cite{CortinesLouidorSaglietti2021,DemboRosenZeitouni2021}.

For random graphs, the cover time was studied in a long series of works by
Cooper and Frieze and their collaborators.  Building on earlier results of
Jonasson~\cite{Jonasson1998}, they determined the first-order asymptotics of
the cover time, with high probability over the graph, for random regular
graphs~\cite{CooperFrieze2005}, for sparse Erd\H{o}s--R\'enyi graphs above the
connectivity threshold~\cite{CooperFrieze2007sparse}, for the giant component
of a supercritical random graph~\cite{CooperFrieze2008giant}, for the
preferential attachment graph~\cite{CooperFrieze2007PA}, for random geometric
graphs~\cite{CooperFrieze2011geometric}, for random digraphs
\cite{CooperFrieze2012digraphs}, and for random graphs with a given degree
sequence~\cite{AbdullahCooperFrieze2012,CooperFriezeLubetzky2014}.  In
particular, for a uniform random $d$-regular graph $G_N$ on $N$ vertices with
$d\ge3$ fixed, Cooper and Frieze~\cite{CooperFrieze2005} proved that, with high
probability,
\begin{equation}
   \max_{x\in V_N}\bbE_x^{G_N}[T_{\rm cov}]=(1+o(1))\,\theta_d\,N\log N,
   \qquad
   \theta_d:=\frac{d-1}{d-2}.
   \label{eq:theta-d}
\end{equation}
The constant $\theta_d$ is the expected number of visits of simple random walk
on the infinite $d$-regular tree $\Td$ to its starting point, reflecting the
fact that $G_N$ is locally tree-like.  A central tool in these works is the
\emph{first visit time lemma}, which expresses the probability that a
stationary walk has not yet visited a given vertex in terms of a
return-probability generating function; a probabilistic proof was given by
Manzo, Quattropani, and Scoppola~\cite{ManzoQuattropaniScoppola2021}.  The
same circle of ideas was used for several independent walks on random regular
graphs by Cooper, Frieze, and Radzik~\cite{CooperFriezeRadzik2010} (see also
Alon, Avin, Kouck\'y, Kozma, Lotker, and Tuttle~\cite{AlonEtAl2011} for
multiple walks on general graphs), and for the structure of the set of
unvisited vertices (the \emph{vacant set}) by Cooper and
Frieze~\cite{CooperFrieze2013,CooperFrieze2016}; see also \v{C}ern\'y,
Teixeira, and Windisch~\cite{CernyTeixeiraWindisch2011} and \v{C}ern\'y and
Teixeira~\cite{CernyTeixeira2013} for the phase transition of the vacant set
on random regular graphs.  On the physics side, Tishby, Biham, and
Katzav~\cite{TishbyBihamKatzav2022} derived an analytical approximation for
the full distribution of the cover time of random regular graphs, predicting
Gumbel fluctuations on the scale $N$.

The purpose of this article is to go beyond the first-order asymptotics
\eqref{eq:theta-d}: we identify the second-order behaviour and the limiting
fluctuations of the cover time of random regular graphs.

For a finite graph $G=(V,E)$ and $x\in V$, we write $\bbP_x^G$ and $\bbE_x^G$
for the law and expectation of discrete-time simple random walk
$(X_t)_{t\ge0}$ on $G$ with $X_0=x$, and $\bbP_{\pi_G}^G$ when $X_0$ is
distributed according to the stationary distribution $\pi_G$.  For $v\in V$
and $A\subseteq V$ let
\begin{equation}
   H_v:=\inf\{t\ge0:X_t=v\},
   \qquad
   H_A:=\inf\{t\ge0:X_t\in A\},
\end{equation}
be the hitting times of $v$ and of $A$, and let
\begin{equation}
   T_{\rm cov}:=\max_{v\in V}H_v
\end{equation}
be the cover time.  Throughout, $G_N$ denotes a uniform simple $d$-regular
graph on $N$ labelled vertices, and $\bbP$ denotes its law.  A statement
about the walk on $G_N$ is said to hold \emph{in probability over $G_N$} if
the corresponding (random) quantity, which is a function of $G_N$, converges
in $\bbP$-probability.  We denote by $\Lambda$ a standard Gumbel random
variable, $\bbP(\Lambda\le s)=e^{-e^{-s}}$ for $s\in\mathbb R$.

For a uniform random $d$-regular graph, $d\ge3$ fixed, we prove the following
result.  It sharpens the first-order result~\eqref{eq:theta-d} of Cooper and
Frieze~\cite{CooperFrieze2005} and confirms the prediction of Tishby, Biham,
and Katzav~\cite{TishbyBihamKatzav2022}.

\begin{theorem}[Second-order cover time of random regular graphs]
\label{thm:rrg-cover-time}
Let $d\ge3$ be fixed, take $N$ along integers for which $dN$ is even, and
let $G_N$ be a uniform simple $d$-regular graph on $N$ labelled vertices.
Conditional on $G_N=G$, let $(X_t)$ be discrete-time simple random walk on
$G$, and let $\theta_d=(d-1)/(d-2)$.  Then, for every $s\in\mathbb R$,
\begin{equation}
   \sup_{x\in V_N}\left|
   \bbP_x^{G_N}\bigl(T_{\rm cov}\le\theta_dN(\log N+s)\bigr)
   -e^{-e^{-s}}\right|\longrightarrow0
   \qquad\text{in probability over }G_N.
\end{equation}
Equivalently, in probability over $G_N$ and uniformly over all starting
points $x\in V_N$, under $\bbP_x^{G_N}$,
\begin{equation}
   \frac{T_{\rm cov}-\theta_dN\log N}{\theta_dN}
   \Rightarrow \Lambda.
\end{equation}
\end{theorem}

In particular, $T_{\rm cov}=\theta_dN\log N+O_{\bbP}(N)$, so that the
second-order term of the cover time is of order $N$ and has no deterministic
correction of intermediate order (such as $N\log\log N$).  This is in contrast
with the two-dimensional torus and the binary tree, where the cover time is
governed by log-correlated fields and the limit
law, where known, is not a pure Gumbel law~\cite{BeliusKistler2017,
CortinesLouidorSaglietti2021,DemboRosenZeitouni2021}.

Random regular graphs are typically not vertex-transitive, so the results
of~\cite{BerestyckiHermonTeyssier2022} do not apply.  Instead, our proof
exploits that, with high probability, $G_N$ is locally tree-like and has an
absolute spectral gap bounded away from zero, by Friedman's
theorem~\cite{Friedman2008} (see also Bordenave~\cite{Bordenave2019}).  In
particular, the random walk mixes in $O(\log N)$ steps (see Lubetzky and
Sly~\cite{LubetzkySly2010} for the cutoff phenomenon), and hitting times of
small sets from stationarity are approximately exponential, in the spirit of
Aldous and Brown~\cite{AldousBrown1992}; refined exponential
approximations are developed in~\cite{BerestyckiHermonTeyssier2026}.

\paragraph{Overview of the proof of Theorem \ref{thm:rrg-cover-time} and the remainder of the article.}
The proof of Theorem~\ref{thm:rrg-cover-time} makes the coupon-collector
heuristic rigorous.  If the events $\{H_v>t\}$ were independent with
$\bbP_\pi(H_v>t)\approx e^{-t/(\theta_dN)}$, the number of unvisited vertices
at time $\theta_dN(\log N+s)$ would be approximately Poisson with mean
$e^{-s}$, and the Gumbel law would follow. As the events $\{H_v>t\}$ are not independent, we instead control all falling factorial moments of the number of
unvisited vertices.  This requires the probability that the walk avoids an
arbitrary set of bounded size up to an error $o(1)$ in the exponent, and
hence the relevant capacities up to an error $o(1/\log N)$.  The proof has three ingredients, developed in
Sections~\ref{sec:fvtl}--\ref{sec:rrg-estimates} and combined in
Sections~\ref{sec:proof-main} and~\ref{sec:deterministic-start}.

Section~\ref{sec:fvtl} works on an arbitrary regular graph with an absolute
spectral gap.  Its main result, Theorem~\ref{thm:uniform-finite-target-fvtl},
states that a stationary walk avoids a set $A$ of bounded size up to time
$t=O(N\log N)$ with probability $\exp\{-t\capGT(A)/N+o(1)\}$, where
$\capGT(A)$ is the capacity of $A$ computed from a Green matrix truncated
at time $T=O(\log N)$.  For a singleton this is the first visit time lemma
of Cooper and Frieze; here we obtain an estimate that is uniform over all
targets of bounded size.

In Section~\ref{sec:tree-capacities}, we state and compute capacities on the infinite
$d$-regular tree $\Td$.  It is well known, that asingle vertex has capacity $\theta_d^{-1}$
(see also Lemma~\ref{lem:tree-basic-capacities}). Every set of at least two vertices has capacity at
least $\frac{2(d-1)}{d}\theta_d^{-1}$ (see Lemma~\ref{lem:tree-cluster-gap}).
Since this larger than $\theta_d^{-1}$, pairs of nearby vertices are much less
likely to be avoided together than a single vertex.  We refer to this mechanism as the
\emph{cluster gap}.

Section~\ref{sec:rrg-estimates} transfers these computations to a random regular graph $G_N$.
With high probability $G_N$ has an absolute spectral gap and is locally
tree-like, apart from $N^{o(1)}$ vertices near short cycles
(Propositions~\ref{prop:typical-rrg-event} and~\ref{prop:diagonal-event}).

In Section~\ref{sec:proof-main} we prove Theorem~\ref{thm:rrg-cover-time} when starting from the stationary distribution.  
In Section~\ref{sec:deterministic-start}, we allow for an arbitrary starting point as 
after $O(\log N)$ steps the walk is close to stationary.  We formulate
this step for general Markov chains
(Proposition~\ref{prop:deterministic-start-comparison}), so that it can be
reused in the later sections.

The remaining sections show that the method extends with little extra
effort.  Section~\ref{sec:lazy-continuous} treats the lazy walk with
holding probability $1-p$, whose cover time is $(\theta_d/p)N(\log
N+\Lambda)$, and the continuous-time walk
(Theorem~\ref{thm:lazy-continuous}).  Section~\ref{sec:parallel-walks}
treats $q$ independent walks, which cover $G_N$ in time
$(\theta_d/q)N(\log N+\Lambda)$ (Theorem~\ref{thm:parallel-walks}).
Section~\ref{sec:late-set-structure} studies the set of vertices that are
still unvisited at times $\theta_dN(\gamma\log N+s)$ with $0<\gamma\le1$.
This set is asymptotically uniformly distributed over the graph
(Theorem~\ref{thm:late-set-equidistribution}).  Its small components
appear in a hierarchy: unvisited copies of each fixed tree with $k$
vertices disappear at an explicit threshold $\gamma_{d,k}$, with Poisson
limits at the threshold (Theorem~\ref{thm:late-component-hierarchy}).  As a
consequence we determine the size of the largest component of the
unvisited set in this regime (Corollary~\ref{cor:late-largest-component}).
\paragraph{Statement on AI usage.} We used Opus 5.5 for proofreading and revising the first draft of
the paper. AI tools were not used to develop mathematical ideas or in the writing the first draft. 
All suggested changes were reviewed by the author, who takes full responsibility for the content of
the paper.
\paragraph{Acknowledgment.} The work of AK was partially supported by the Cusanuswerk. Part of this work was written while AK visited Marek Biskup and Johannes Bäumler at UCLA. He thanks them for their hospitality and for early discussions on the topic.  The work of LH was partly funded by the Deutsche Forschungsgemeinschaft (DFG, German Research
Foundation) through Project-ID 443891315 within SPP 2265.
\section{Uniform finite-target first-visit estimates}
\label{sec:fvtl}

In this section $G$ is a fixed graph satisfying
Assumption~\ref{ass:spectral-gap} below; nothing here is specific to random
graphs.  The main result, Theorem~\ref{thm:uniform-finite-target-fvtl},
expresses the probability that a stationary walk avoids a set of bounded
size up to time $O(N\log N)$ in terms of a truncated Green-matrix capacity.
Its proof has two steps: the avoidance probability is governed by the
killing rate of a quasi-stationary distribution
(Lemma~\ref{lem:quasi-stationary-exponential}), and this rate equals the
capacity divided by $N$ up to an error $O_{K,\rho}(\log N/N^2)$
(Lemma~\ref{lem:entrance-rate-green-capacity}).

\begin{assumption}[Absolute spectral gap]
\label{ass:spectral-gap}
$G=(V,E)$ is a connected $d$-regular graph with $|V|=N$, $P=P_G$ is the
transition matrix of simple random walk on $G$, and all eigenvalues of $P$
other than $1$ lie in $[-\rho,\rho]$ for some fixed $\rho<1$.
\end{assumption}

Under Assumption~\ref{ass:spectral-gap}, $P$ is symmetric and the
stationary distribution $\pi$ is uniform.

\begin{lemma}[Mixing]
\label{lem:mixing}
Under Assumption~\ref{ass:spectral-gap}, for all $x,y\in V$ and $t\ge0$,
\begin{equation}
   \Bigl|P^t(x,y)-\frac1N\Bigr|\le\rho^t.
   \label{eq:pointwise-mixing}
\end{equation}
Moreover, the total-variation mixing function
$\eta_N(m):=\max_{x\in V}\|P^m(x,\cdot)-\pi\|_{\TV}$ satisfies
\begin{equation}
   \eta_N(m)\le\tfrac12\sqrt N\rho^m,
   \label{eq:tv-mixing}
\end{equation}
and for every $f:V\to[0,1]$, $x\in V$ and $m\ge0$,
\begin{equation}
   \bigl|\bbE_x[f(X_m)]-\bbE_\pi[f]\bigr|\le\eta_N(m).
   \label{eq:tv-test-function}
\end{equation}
\end{lemma}

\begin{proof}
Let $\varphi_1=N^{-1/2}\one,\varphi_2,\dots,\varphi_N$ be an orthonormal
eigenbasis of $P$ with eigenvalues $1,\xi_2,\dots,\xi_N$, so that
$|\xi_i|\le\rho$ for $i\ge2$.  Then
$P^t(x,y)-1/N=\sum_{i\ge2}\xi_i^t\varphi_i(x)\varphi_i(y)$, and
\eqref{eq:pointwise-mixing} follows from the Cauchy--Schwarz inequality and
$\sum_i\varphi_i(x)^2=1$.  Similarly
$\sum_y(P^m(x,y)-1/N)^2=\sum_{i\ge2}\xi_i^{2m}\varphi_i(x)^2\le\rho^{2m}$,
and Cauchy--Schwarz gives
$\sum_y|P^m(x,y)-1/N|\le\sqrt N\rho^m$, which is \eqref{eq:tv-mixing}.
Finally, since $\sum_y(P^m(x,y)-\pi(y))=0$, the quantity
$\bbE_x[f(X_m)]-\bbE_\pi[f]=\sum_y(P^m(x,y)-\pi(y))f(y)$ lies between
$-\sum_y(P^m(x,y)-\pi(y))_-$ and $\sum_y(P^m(x,y)-\pi(y))_+$, and both
sums equal $\|P^m(x,\cdot)-\pi\|_{\TV}$.
\end{proof}

\paragraph{Truncated Green matrices and capacities.}
Throughout the paper the truncation level is
\begin{equation}
   T=T_N:=\Bigl\lceil\frac{\log N}{\log(1/\rho)}\Bigr\rceil,
   \qquad\text{so that}\qquad
   \rho^T\le\frac1N
   \quad\text{and}\quad T=O(\log N).
   \label{eq:truncation-level}
\end{equation}
For $x,y\in V$ and $A\subseteq V$ let
\begin{equation}
   g_T^G(x,y):=\sum_{j=0}^{T}P^j(x,y),
   \qquad
   G_T^G(A):=\bigl(g_T^G(a,b)\bigr)_{a,b\in A},
   \qquad
   \capGT(A):=\one_A^\top\big(G_T^G(A)\big)^{-1}\one_A.
   \label{eq:truncated-capacity}
\end{equation}
By Lemma~\ref{lem:uniform-green-invertibility} below, $G_T^G(A)$ is
invertible, so $\capGT(A)$ is well defined.

\begin{theorem}[Uniform finite-target first-visit theorem]
\label{thm:uniform-finite-target-fvtl}
Fix $K<\infty$ and $C_0<\infty$.  Under Assumption~\ref{ass:spectral-gap},
uniformly over $A\subseteq V$ with $1\le |A|\le K$ and integers
$0\le t\le C_0N\log N$,
\begin{equation}
   \bbP_\pi(H_A>t)
   =
   \exp\left\{-\frac{t}{N}\capGT(A)
   +O_{K,C_0,\rho}\!\left(\frac{(\log N)^2}{N}\right)\right\}.
\end{equation}
The error bound is deterministic and uniform in the graph, the target
$A$, and the indicated integer time $t$.  In particular, the error in the
exponent is $o(1)$.
\end{theorem}

\begin{remark}[Relation to the first visit time lemma]
\label{rem:fvtl-singleton}
For a singleton, $\capGT(\{v\})=1/g_T^G(v,v)$, where $g_T^G(v,v)$ is the
expected number of visits to $v$ during the first $T$ steps of the walk
started at $v$.  Theorem~\ref{thm:uniform-finite-target-fvtl} then reads
$\bbP_\pi(H_v>t)=\exp\{-t/(Ng_T^G(v,v))+o(1)\}$, which is the form of
Cooper and Frieze's first visit time lemma~\cite{CooperFrieze2005}; see
also~\cite{ManzoQuattropaniScoppola2021}.  The point of
Theorem~\ref{thm:uniform-finite-target-fvtl} is that it holds for all
targets of bounded size simultaneously, with the matrix
$G_T^G(A)$ encoding the interaction between the points of $A$.
\end{remark}

\begin{lemma}[Truncated Green matrices]
\label{lem:uniform-green-invertibility}
Under Assumption~\ref{ass:spectral-gap}, for every $A\subseteq V$ the
matrix $G_T^G(A)$ is positive definite with smallest eigenvalue at least
$c_\rho:=(1-\rho)/(1+\rho)$; in particular
$\|(G_T^G(A))^{-1}\|_{\op}\le c_\rho^{-1}$.  Moreover, for $N\ge T$ and all
$x,y\in V$,
\begin{equation}
   g_T^G(x,y)\le\one_{\{x=y\}}+\frac TN+\frac{\rho}{1-\rho}
   \le M_\rho:=2+\frac{\rho}{1-\rho}.
   \label{eq:green-entry-bound}
\end{equation}
\end{lemma}

\begin{proof}
The matrix $G_T^G(V)=\sum_{j=0}^TP^j$ is symmetric, and its eigenvalues
are $S_T(\xi):=\sum_{j=0}^T\xi^j$ for the eigenvalues $\xi$ of $P$.  We have
$S_T(1)=T+1$, $S_T(\xi)\ge1$ for $\xi\in[0,\rho]$, and, for
$\xi\in[-\rho,0)$, $S_T(\xi)=(1-\xi^{T+1})/(1-\xi)\ge(1-\rho)/(1+\rho)$.
Hence the smallest eigenvalue of $G_T^G(V)$ is at least $c_\rho$, and by
Cauchy interlacing the same holds for every principal submatrix
$G_T^G(A)$.  The entry bound \eqref{eq:green-entry-bound} follows by
summing \eqref{eq:pointwise-mixing} over $1\le j\le T$.
\end{proof}

\paragraph{Quasi-stationary distribution.}
Let $A\subseteq V$ and $B=A^c$.  Let $Q:=P|_{B\times B}$ be the
transition matrix of the walk killed upon entering $A$, and write
$\lambda_A=1-r_A$ for the largest eigenvalue of $Q$.

\begin{lemma}
\label{lem:quasi-stationary-exponential}
Fix $K<\infty$ and $C_0<\infty$.  Under Assumption~\ref{ass:spectral-gap},
uniformly over $A\subseteq V$ with $1\le |A|\le K$ and integers
$0\le t\le C_0N\log N$,
\begin{equation}
   \begin{aligned}
      0<r_A&=O_K(1/N),\\
      \bbP_\pi(H_A>t)
      &=\lambda_A^t\bigl(1+O_{K,\rho}(N^{-1})\bigr)\\
      &=\exp\left\{-tr_A+O_{K,C_0,\rho}\!\left(\frac{\log N}{N}\right)\right\}.
   \end{aligned}
\end{equation}
Moreover, for all sufficiently large $N$, $\lambda_A>\rho$ is a simple
eigenvalue of $Q$ and all other eigenvalues of $Q$ lie in $[-\rho,\rho]$.
\end{lemma}

\begin{proof}
Let $m=|A|$ and $u_B:=\one_B/\sqrt{|B|}$.  Since at most $d|A|$ directed
edges enter $A$, $\sum_{x\in B}P(x,A)\le|A|\le K$, so by the Rayleigh
quotient
\begin{equation}
   \lambda_A\ge u_B^\top Qu_B
   =1-\frac1{|B|}\sum_{x\in B}P(x,A)
   \ge1-O_K(1/N).
\end{equation}
Since $Q$ is substochastic, $\lambda_A\le1$, and in fact $\lambda_A<1$:
every connected component of $B$ has an edge leading to $A$, so the walk
killed on that component loses mass with positive probability.  This proves
$0<r_A=O_K(1/N)$.

Let $1=\xi_1\ge\cdots\ge\xi_N$ be the eigenvalues of $P$ and
$\mu_1\ge\cdots\ge\mu_{N-m}$ those of $Q$.  Cauchy interlacing gives
$\xi_i\ge\mu_i\ge\xi_{i+m}$, so $\mu_2\le\xi_2\le\rho$ and
$\mu_{N-m}\ge\xi_N\ge-\rho$.  For large $N$, $\lambda_A=\mu_1>\rho$, which
proves the last assertion.

Let $\varphi_1,\dots,\varphi_{N-m}$ be an orthonormal eigenbasis of $Q$
with $Q\varphi_1=\lambda_A\varphi_1$, and expand
$u_B=\alpha\varphi_1+\sum_{i\ge2}\beta_i\varphi_i$, so that
$\alpha^2+\sum_{i\ge2}\beta_i^2=1$.  Then
\begin{equation}
   1-O_K(1/N)\le u_B^\top Qu_B
   =\alpha^2\lambda_A+\sum_{i\ge2}\beta_i^2\mu_i
   \le\alpha^2+(1-\alpha^2)\rho,
\end{equation}
so $1-\alpha^2=O_{K,\rho}(1/N)$.  Since
$\bbP_\pi(H_A>t)=N^{-1}\one_B^\top Q^t\one_B=(|B|/N)\,u_B^\top Q^tu_B$, the
spectral expansion gives
\begin{equation}
   \begin{aligned}
      \bbP_\pi(H_A>t)
      &=\lambda_A^t\bigl(1+O_{K,\rho}(1/N)\bigr)
        +O_{K,\rho}(\rho^t/N)\\
      &=\lambda_A^t\bigl(1+O_{K,\rho}(1/N)\bigr),
   \end{aligned}
\end{equation}
where the last step uses $\rho<\lambda_A$, so that
$(\rho/\lambda_A)^t\le1$.  Finally,
$\lambda_A^t=(1-r_A)^t=\exp\{-tr_A+O(tr_A^2)\}$ and
$tr_A^2=O_{K,C_0}(\log N/N)$.  Taking the logarithm of the relative
factor adds only $O_{K,\rho}(1/N)$, which proves the stated exponential
estimate.
\end{proof}

For $|A|\le K$ and $N$ large, the Perron--Frobenius theorem gives a
nonnegative eigenvector of $Q$ for its largest eigenvalue $\lambda_A$.
Since $Q$ is symmetric and this eigenvalue is simple by
Lemma~\ref{lem:quasi-stationary-exponential}, normalization gives a unique
probability row vector $\nu_A$ satisfying $\nu_AQ=\lambda_A\nu_A$.
Thus $\nu_A$ is the quasi-stationary distribution associated with the
leading eigenvalue.  We define the \emph{entrance measure} into $A$ by
\begin{equation}
   \zeta_A:=\nu_AP_{B,A},
   \qquad\text{so that}\qquad
   \zeta_A\one_A=\nu_A(\one_B-Q\one_B)=1-\lambda_A=r_A.
   \label{eq:entrance-measure}
\end{equation}

\begin{lemma}
\label{lem:entrance-rate-green-capacity}
Fix $K<\infty$.  Under Assumption~\ref{ass:spectral-gap}, uniformly over
$A\subseteq V$ with $1\le |A|\le K$,
\begin{equation}
   r_A=\frac1N\capGT(A)
   +O_{K,\rho}\!\left(\frac{\log N}{N^2}\right).
\end{equation}
\end{lemma}

\begin{proof}
Fix $b\in A$ and set $n:=2T+1$.  Starting from $\nu_A$, survival for $s-1$
steps has probability $\lambda_A^{s-1}$ and leaves the walk distributed as
$\nu_A$, so
$\bbP_{\nu_A}(H_A=s,X_s=a)=\lambda_A^{s-1}\zeta_A(a)$.  Decomposing
according to the first entrance into $A$ (which must occur by time $n$,
since $b\in A$) and substituting $u=n-s$,
\begin{equation}
   \bbP_{\nu_A}(X_n=b)
   =\sum_{u=0}^{2T}\lambda_A^{2T-u}\sum_{a\in A}\zeta_A(a)P^u(a,b).
   \label{eq:entrance-convolution}
\end{equation}
By \eqref{eq:pointwise-mixing} and \eqref{eq:truncation-level}, the
left-hand side equals $1/N+O(\rho^n)=1/N+O(N^{-2})$.

We compare the right-hand side with $(\zeta_AG_T^G(A))(b)$.  For
$0\le u\le T$, $\lambda_A^{2T-u}=1+O(Tr_A)=1+O_{K,\rho}(\log N/N)$; since
$g_T^G\le M_\rho$ by \eqref{eq:green-entry-bound} and
$\zeta_A\one_A=r_A=O_K(1/N)$, replacing these factors by $1$ costs
$O_{K,\rho}(\log N/N^2)$.  For $T<u\le2T$, \eqref{eq:pointwise-mixing} and
$\rho^T\le1/N$ give
\begin{equation}
   \sum_{u=T+1}^{2T}\sum_{a\in A}\zeta_A(a)P^u(a,b)
   \le r_A\sum_{u=T+1}^{2T}\Bigl(\frac1N+\rho^u\Bigr)
   \le O_K\Bigl(\frac1N\Bigr)\Bigl(\frac TN+\frac{1}{(1-\rho)N}\Bigr)
   =O_{K,\rho}\Bigl(\frac{\log N}{N^2}\Bigr).
\end{equation}
Combining these estimates gives
$\zeta_AG_T^G(A)=N^{-1}\one_A^\top+\varepsilon_A$ with
$\|\varepsilon_A\|_\infty=O_{K,\rho}(\log N/N^2)$ uniformly in $A$.
Multiplying on
the right by $(G_T^G(A))^{-1}$, using
Lemma~\ref{lem:uniform-green-invertibility} and $|A|\le K$, and summing over
$A$ with \eqref{eq:entrance-measure} gives the claim.
\end{proof}

\begin{proof}[Proof of Theorem~\ref{thm:uniform-finite-target-fvtl}]
By Lemma~\ref{lem:quasi-stationary-exponential},
\begin{equation}
   \log\bbP_\pi(H_A>t)
   =-tr_A+O_{K,C_0,\rho}(\log N/N).
\end{equation}
Lemma~\ref{lem:entrance-rate-green-capacity} gives
\begin{equation}
   tr_A=\frac{t}{N}\capGT(A)
   +O_{K,C_0,\rho}\!\left(\frac{(\log N)^2}{N}\right),
\end{equation}
uniformly for $1\le |A|\le K$ and $0\le t\le C_0N\log N$.
Combining the two estimates proves the theorem.
\end{proof}

\section{Capacities on the infinite regular tree}
\label{sec:tree-capacities}

This section computes the tree capacities that appear in
Section~\ref{sec:rrg-estimates} as local limits of the finite-graph
capacities $\capGT$.  The one-point capacity
(Lemma~\ref{lem:tree-basic-capacities}) gives the cover-time constant
$\theta_d$; the lower bound for sets of at least two points
(Lemma~\ref{lem:tree-cluster-gap}), which we call the \emph{cluster gap},
is what makes close tuples negligible.

Let $\Td$ be the infinite $d$-regular tree, $d\ge3$.  For simple random
walk on $\Td$, define the Green function and, for nonempty finite
$F\subset\Td$, the Green matrix and capacity
\begin{equation}
   g_{\Td}(x,y):=\sum_{t=0}^{\infty}\bbP_x(X_t=y),
   \qquad
   G_{\Td}(F):=\big(g_{\Td}(x,y)\big)_{x,y\in F},
   \qquad
   \Cap_{\Td}(F):=\one_F^\top G_{\Td}(F)^{-1}\one_F.
\end{equation}

\begin{lemma}
\label{lem:tree-green-function}
For all $x,y\in\Td$,
\begin{equation}
   g_{\Td}(x,y)=\theta_d(d-1)^{-\dist_{\Td}(x,y)}
\end{equation}
The Green operator $G_{\Td}:=(I-P_{\Td})^{-1}$ is bounded and strictly
positive on $\ell^2(\Td)$, and for every nonempty finite $F\subset\Td$
the matrix $G_{\Td}(F)$ is positive definite with
\begin{equation}
   \|G_{\Td}(F)^{-1}\|_{\op}\le1+\beta_d,
\end{equation}
where $\beta_d=\|P_{\Td}\|_{2\to2}=\frac{2\sqrt{d-1}}d<1$ is the spectral radius.
\end{lemma}

\begin{proof}
These are classical results; see for example
Woess~\cite{Woess2000} for the Green function of $\Td$ and
Kesten~\cite{Kesten1959} for the spectral radius $\beta_d$.
\end{proof}

The next lemma is needed twice: the singleton capacity determines the scale, and the two-point formula gives the cluster gap.

\begin{lemma}
\label{lem:tree-basic-capacities}
The one-point capacity is
\begin{equation}\label{re.5}
   \Cap_{\Td}(\{o\})=\theta_d^{-1}.
\end{equation}
If $x\ne y$ and $r=\dist_{\Td}(x,y)$, then
\begin{equation}
   \Cap_{\Td}(\{x,y\})=\frac{2\theta_d^{-1}}{1+(d-1)^{-r}}.
\end{equation}
In particular,
\begin{equation}
   \Cap_{\Td}(\{x,y\})\ge \frac{2(d-1)}d\theta_d^{-1}.
\end{equation}
\end{lemma}

\begin{proof}
By Lemma~\ref{lem:tree-green-function},
$\Cap_{\Td}(\{o\})=g_{\Td}(o,o)^{-1}=\theta_d^{-1}$,
which proves \eqref{re.5}.
For the pair, Lemma~\ref{lem:tree-green-function} gives
\begin{equation}
   G_{\Td}(\{x,y\})=\theta_d
   \begin{pmatrix}
      1 & a\\
      a & 1
   \end{pmatrix},
   \qquad
   a=(d-1)^{-r}.
\end{equation}
Inverting this matrix and using the definition of $\Cap_{\Td}(\{x,y\})$ gives
\begin{equation}
   \Cap_{\Td}(\{x,y\})
   =\frac{2(1-a)}{\theta_d(1-a^2)}
   =\frac{2\theta_d^{-1}}{1+(d-1)^{-r}}.
\end{equation}
Since $r\ge1$,
$
   1+(d-1)^{-r}\le \frac d{d-1},
$
and the lower bound follows.
\end{proof}

The following lemma identifies the matrix definition with the usual network
capacity.  Besides making the normalization explicit, it supplies the
monotonicity needed to extend the two-point lower bound to arbitrary clusters.

\begin{lemma}[Variational formula, monotonicity and cluster gap]
\label{lem:tree-cluster-gap}
For every nonempty finite $F\subset\Td$,
\begin{equation}\label{re.6}
   \Cap_{\Td}(F)
   =\inf\left\{\mathcal E_{\Td}(f,f):
   f\in C_c(\Td),\ f\equiv1\text{ on }F\right\},
\end{equation}
where the edge conductance is $1/d$ and
\begin{equation}
   \mathcal E_{\Td}(f,f)
   :=\frac12\sum_{x\in\Td}\sum_{y\sim x}\frac1d
          \bigl(f(x)-f(y)\bigr)^2.
\end{equation}
If $A\subseteq B\subset\Td$ are finite, then
\begin{equation}
   \Cap_{\Td}(A)\le \Cap_{\Td}(B).
\end{equation}
Consequently, every finite $B\subset\Td$ with $|B|\ge2$ satisfies
\begin{equation}
   \Cap_{\Td}(B)\ge \frac{2(d-1)}d\theta_d^{-1}
  .
\end{equation}
\end{lemma}

\begin{proof}
Identity \eqref{re.6} is the Dirichlet principle for the reversible walk
on $\Td$ with conductances $1/d$; see for example
\cite[Chapter~2]{LyonsPeres2016}.  In the present normalization,
$e_F:=G_{\Td}(F)^{-1}\one_F$ is the equilibrium measure,
$e_F(x)=\bbP_x(X_t\notin F\text{ for all }t\ge1)$ for $x\in F$, by
the last-exit decomposition $\sum_{y\in F}g_{\Td}(x,y)e_F(y)=1$ for
$x\in F$; thus $\Cap_{\Td}(F)=\sum_{x\in F}e_F(x)$.
If $A\subseteq B$, every function admissible for $B$ in \eqref{re.6} is
also admissible for $A$; taking infima gives
$\Cap_{\Td}(A)\le\Cap_{\Td}(B)$.

If $|B|\ge2$, choose distinct $x,y\in B$.  By monotonicity and Lemma~\ref{lem:tree-basic-capacities},
\begin{equation}
   \Cap_{\Td}(B)
   \ge \Cap_{\Td}(\{x,y\})
   \ge \frac{2(d-1)}d\theta_d^{-1}.
\end{equation}
\end{proof}

\section{Random regular graph estimates}
\label{sec:rrg-estimates}

This section collects the estimates on a typical random regular graph that
are needed in Section~\ref{sec:proof-main}.  Propositions~\ref{prop:typical-rrg-event}
and~\ref{prop:diagonal-event} describe the typical graph event.  The main
result is the local comparison, Proposition~\ref{prop:local-green-comparison},
which turns finite-graph capacities of targets away from short cycles into
tree capacities with precision $o(1/\log N)$.  Targets near short cycles are
handled by the cruder Proposition~\ref{prop:rough-capacity-bound}, and
Lemma~\ref{lem:cluster-profiles} combines both into the form used later.

\paragraph{Parameters.}
Throughout this section $G_N$ is a uniform simple $d$-regular graph on $N$
labelled vertices, $\rho<1$ is the constant of \textup{(T2)} below, and
$T=T_N$ is the truncation level \eqref{eq:truncation-level}.  We fix
$0<\alpha<1$ and set
\begin{equation}
   L_N:=\big\lfloor(\log N)^\alpha\big\rfloor .
\end{equation}
Targets are split into clusters on the scale $2L_N$.  Since $\alpha<1$,
balls of radius $O(L_N)$ contain $N^{o(1)}$ vertices, while
$\rho^{L_N}$, $\beta_d^{L_N}$ and $T/N$ are $o(1/\log N)$.  For $K\in\mathbb N$, an
upper bound for the size of target sets, define the set of locally bad
vertices
\begin{equation}
   \mathcal B_{N,K}:=\{v\in V_N:B_G(v,10KL_N)\text{ is not a tree}\}.
\end{equation}
When we work with $k$-tuples we always take $K=k$.

\begin{proposition}
\label{prop:typical-rrg-event}
Fix $d\ge3$ and $K<\infty$.  With probability tending to one, $G_N$ satisfies the following properties.
\begin{enumerate}[label=\textup{(T\arabic*)}]
\item $G_N$ is connected and non-bipartite.
\item There is $\rho<1$, depending only on $d$, such that every nontrivial eigenvalue of the random-walk transition matrix lies in $[-\rho,\rho]$.
\item The number of locally bad vertices satisfies
\begin{equation}
   |\mathcal B_{N,K}|=N^{o(1)}.
\end{equation}
\item For every fixed $k\ge2$, the number of ordered $k$-tuples with some pair at distance at most $2L_N$ is at most
\begin{equation}
   N^{k-1+o(1)}.
\end{equation}
For $k=1$ this number is zero.
\end{enumerate}
\end{proposition}

\begin{proof}
Friedman's theorem~\cite{Friedman2008} (see also~\cite{Bordenave2019})
gives that, with probability tending to one, all adjacency eigenvalues other
than $d$ are at most $2\sqrt{d-1}+o(1)$ in absolute value.  Dividing by $d$
gives \textup{(T2)} for any fixed $\rho\in(2\sqrt{d-1}/d,1)$.  Property
\textup{(T1)} is a consequence of \textup{(T2)}: the eigenvalue $1$ is
simple, so $G_N$ is connected, and $-1$ is not an eigenvalue, so $G_N$ is
not bipartite.  We list it separately only for later reference.

For local tree-likeness, use the configuration model~\cite{Bollobas1980,Wormald1999}.  The radius-$R$ exploration from a fixed vertex has size at most $\exp(O(R))$.  The probability that this exploration creates a collision is at most
\begin{equation}
   \frac{\exp(O(R))}{N}.
\end{equation}
With $R=10KL_N$, the expected number of vertices with non-tree $R$-balls is
\begin{equation}
   N\cdot \frac{\exp(O_K(L_N))}{N}=\exp(O_K(L_N))=N^{o(1)}.
\end{equation}
Markov's inequality gives $|\mathcal B_{N,K}|=N^{o(1)}$ with high probability.  Finally, a ball of radius $2L_N$ has size at most
\begin{equation}
   \exp(O(L_N))=N^{o(1)}.
\end{equation}
Thus, for $k\ge2$, after choosing one vertex and one close partner, the
remaining $k-2$ vertices are arbitrary, which gives the asserted upper
bound.  There is no close-pair tuple when $k=1$.
The probability that the configuration model is simple converges to a
strictly positive constant for fixed $d$~\cite{Bollobas1980,Wormald1999}.  Therefore each $o(1)$ failure
probability above remains $o(1)$ after conditioning on simplicity, which
transfers the estimates to the uniform simple model.
\end{proof}

The preceding proposition is stated at each fixed target size.  The following
standard diagonalization produces one graph event on which every fixed size is
eventually available; this is the event used below.

\begin{proposition}
\label{prop:diagonal-event}
There are integers $K_N\uparrow\infty$ and graph events $\mathcal E_N$ with
$\bbP(G_N\in\mathcal E_N)\to1$ such that, on $\mathcal E_N$,
properties \textup{(T1)}--\textup{(T2)} hold and, simultaneously for every
$1\le K,k\le K_N$,
\begin{equation}
   |\mathcal B_{N,K}|\le N^{1/K_N}
\end{equation}
and, for $k\ge2$, the number of ordered $k$-tuples containing a pair at
distance at most $2L_N$ is at most
\begin{equation}
   N^{k-1+1/K_N}.
\end{equation}
For $k=1$ the latter count is zero.  In particular, for every fixed $K$ and
$k$, these bounds are $N^{o(1)}$ and $N^{k-1+o(1)}$, respectively.
\end{proposition}

\begin{proof}
For an integer $M\ge2$, let $\mathcal E_N(M)$ be the event that
\textup{(T1)}--\textup{(T2)} hold, that
$|\mathcal B_{N,K}|\le N^{1/M}$ for every $K\le M$, and that the close-tuple
count is at most $N^{k-1+1/M}$ for every $2\le k\le M$.  The proof of
Proposition~\ref{prop:typical-rrg-event}, followed by a finite union bound,
shows that $\bbP(\mathcal E_N(M))\to1$ for each fixed $M$.  Choose $N_M$
recursively so that it is strictly increasing, $N_M\ge M$, and
\begin{equation}
   \bbP(\mathcal E_N(M))\ge1-M^{-1}
   \qquad(N\ge N_M),
\end{equation}
and define $K_N$ to be the largest $M$ with $N_M\le N$ (changing finitely
many initial values if necessary).  Then $K_N\to\infty$ and
$\mathcal E_N:=\mathcal E_N(K_N)$ has the required properties.
\end{proof}

We next compare finite-graph Green entries with tree Green entries.  The
precision $o(1/\log N)$ is what the second-order window requires, since
capacities are multiplied by $t_N(s)/N=\theta_d(\log N+s)$.  No tree
comparison is asserted for targets meeting $\mathcal B_{N,K}$; these are
handled by Proposition~\ref{prop:rough-capacity-bound}.

\begin{proposition}
\label{prop:local-green-comparison}
Work on the event of Proposition~\ref{prop:typical-rrg-event}.  Fix $K<\infty$.  Let $A\subseteq V_N$ with $|A|\le K$, and  $A\cap \mathcal B_{N,K}=\emptyset$.  Decompose $A=A_1\cup\cdots\cup A_m$ into $2L_N$-clusters, meaning equivalence classes generated by the relation
\begin{equation}
   x\sim y
   \quad\Longleftrightarrow\quad
   \dist_G(x,y)\le2L_N.
\end{equation}
Then each $A_i$ is identified with a finite set $A_i^*\subset\Td$, and
\begin{equation}
   \capGT(A)=\sum_{i=1}^m\Cap_{\Td}(A_i^*)+o(1/\log N)
\end{equation}
uniformly over all such $A$.
\end{proposition}

\begin{proof}
We first explain the identification $A_i\mapsto A_i^*$.  A cluster contains
at most $K$ vertices and is generated by links of length at most $2L_N$, so
$\dist_G(x,y)\le2(K-1)L_N$ for $x,y\in A_i$.  Fix $x\in A_i$.  Since
$x\notin\mathcal B_{N,K}$, the ball $B_G(x,10KL_N)$ is a tree in which all
vertices at distance less than $10KL_N$ from $x$ have degree $d$, so there
is a graph isomorphism $\psi_x$ from $B_G(x,10KL_N)$ onto the ball of the
same radius around a root of $\Td$.  We set $A_i^*:=\psi_x(A_i)$ and
$y^*:=\psi_x(y)$ for $y\in A_i$; since $\psi_x$ preserves distances,
$G_{\Td}(A_i^*)$ does not depend on the choice of $x$.

We now compare entries of the Green matrix within $A_i$.  Every length-$j$ walk contributing to $P_G^j(x,y)$ stays inside this ball when $j\le4KL_N$.  Hence the transition probabilities agree with those in the corresponding tree neighborhood:
\begin{equation}
   P_G^j(x,y)=P_{\Td}^j(x^*,y^*).
\end{equation}
Therefore
\begin{equation}
   \sum_{j=0}^{4KL_N}P_G^j(x,y)
   =
   \sum_{j=0}^{4KL_N}P_{\Td}^j(x^*,y^*).
\end{equation}
The finite-graph tail is bounded by the spectral gap:
\begin{align}
   \sum_{j>4KL_N}^{T}P_G^j(x,y)
   &\le \sum_{j>4KL_N}^{T}\left(\frac1N+\rho^j\right) \\
   &\le \frac{T}{N}+C\rho^{4KL_N} \\
   &=o(1/\log N).
\end{align}
The tree tail is also $o(1/\log N)$, since
$P_{\Td}^j(x,y)\le\beta_d^j$ and $\beta_d<1$.  Hence
\begin{equation}
   g_T^G(x,y)=g_{\Td}(x^*,y^*)+o(1/\log N)
\end{equation}
within clusters.

If $x\in A_i$ and $y\in A_j$ with $i\ne j$, then by definition of the cluster decomposition
\begin{equation}
   \dist_G(x,y)>2L_N.
\end{equation}
Thus $P_G^j(x,y)=0$ for $j\le2L_N$, and the spectral gap gives
\begin{align}
   g_T^G(x,y)
   &\le \sum_{j>2L_N}^{T}\left(\frac1N+\rho^j\right) \\
   &\le \frac{T}{N}+C\rho^{2L_N} \\
   &=o(1/\log N).
\end{align}
Thus, after ordering the vertices of $A$ by clusters,
\begin{equation}
   G_T^G(A)=\diag\big(G_{\Td}(A_1^*),\dots,G_{\Td}(A_m^*)\big)+R_N,
   \qquad
   \|R_N\|_{\op}=o(1/\log N).
\end{equation}
The matrices have uniformly bounded inverse norms by Lemma~\ref{lem:uniform-green-invertibility} on the finite graph and by Lemma~\ref{lem:tree-green-function} on the tree.  The resolvent identity gives
\begin{equation}
   \big(G_T^G(A)\big)^{-1}
   =
   \diag\big(G_{\Td}(A_1^*)^{-1},\dots,G_{\Td}(A_m^*)^{-1}\big)
   +o(1/\log N).
\end{equation}
Multiplying by $\one_A$ on both sides yields
\begin{equation}
   \capGT(A)=\sum_{i=1}^m\Cap_{\Td}(A_i^*)+o(1/\log N).
\end{equation}
\end{proof}

For targets that meet $\mathcal B_{N,K}$ we only need a positive lower
bound, which holds without any tree comparison.

\begin{proposition}
\label{prop:rough-capacity-bound}
On the event of Proposition~\ref{prop:typical-rrg-event}, for every fixed
$K$ there is $\delta_{d,K}>0$ such that every nonempty $A\subseteq V_N$ with
$|A|\le K$ satisfies
\begin{equation}
   \capGT(A)\ge\delta_{d,K}.
\end{equation}
Moreover, if $A=A_1\cup\cdots\cup A_m$ is its $2L_N$-cluster decomposition, then
\begin{equation}
   \capGT(A)\ge \sum_{i=1}^m\capGT(A_i)-o(1).
\end{equation}
\end{proposition}

\begin{proof}
By \eqref{eq:green-entry-bound}, $g_T^G(x,y)\le M_\rho$ for all
$x,y\in V_N$.  Thus every row sum of $G_T^G(A)$ is bounded by $KM_\rho$, and
the largest eigenvalue of $G_T^G(A)$ is at most $KM_\rho$.  Since the matrix is positive definite,
\begin{equation}
   \capGT(A)=\one_A^\top\big(G_T^G(A)\big)^{-1}\one_A
   \ge \frac{|A|^2}{\one_A^\top G_T^G(A)\one_A}
   \ge \frac1{KM_\rho}.
\end{equation}
This proves the lower bound with $\delta_{d,K}=(KM_\rho)^{-1}$.

For the block estimate, if $x\in A_i$ and $y\in A_j$ with $i\ne j$, then $\dist_G(x,y)>2L_N$.  As in the proof of Proposition~\ref{prop:local-green-comparison},
\begin{equation}
   g_T^G(x,y)=o(1)
\end{equation}
uniformly.  Thus $G_T^G(A)$ is a block diagonal matrix plus an $o(1)$ perturbation.  Uniform inverse bounds and the resolvent identity imply
\begin{equation}
   \capGT(A)=\sum_{i=1}^m\capGT(A_i)+o(1),
\end{equation}
which is stronger than the claimed lower bound.
\end{proof}

\paragraph{Cluster profiles.}
The last result of this section combines Propositions~\ref{prop:local-green-comparison}
and~\ref{prop:rough-capacity-bound} into the form used in
Sections~\ref{sec:proof-main}, \ref{sec:parallel-walks}
and~\ref{sec:late-set-structure}.  Fix $k\ge1$.  For an ordered $k$-tuple
$\mathbf v$ of distinct vertices, let $A(\mathbf v)=A_1\cup\cdots\cup A_m$
be the $2L_N$-cluster decomposition of $A(\mathbf v):=\{v_1,\dots,v_k\}$.
A cluster is \emph{bad} if it meets $\mathcal B_{N,k}$ and \emph{clean}
otherwise.  The \emph{profile} of $\mathbf v$ is $(a_1,a_2,b)$, where $a_1$
is the number of clean singleton clusters, $a_2$ the number of clean
non-singleton clusters and $b$ the number of bad clusters.  We call
$\mathbf v$ \emph{good} if its profile is $(k,0,0)$, that is, if
$v_i\notin\mathcal B_{N,k}$ for all $i$ and $\dist_G(v_i,v_j)>2L_N$ for all
$i\ne j$.

\begin{lemma}[Cluster profiles]
\label{lem:cluster-profiles}
Let $G\in\mathcal E_N$ and fix $k\ge1$.  Uniformly over profiles
$(a_1,a_2,b)$ and over tuples with that profile:
\begin{enumerate}[label=\textup{(\roman*)}]
\item the number of ordered $k$-tuples of distinct vertices with profile
$(a_1,a_2,b)$ is at most $N^{a_1+a_2+o(1)}$, and the number of good tuples
is $(1+o(1))N^k$;
\item every tuple $\mathbf v$ with profile $(a_1,a_2,b)$ satisfies
\begin{equation}
   \capGT(A(\mathbf v))
   \ge a_1\theta_d^{-1}+a_2\frac{2(d-1)}{d}\theta_d^{-1}
      +b\delta_{d,k}-o(1);
   \label{eq:profile-capacity}
\end{equation}
\item every good tuple satisfies
$\capGT(A(\mathbf v))=k\theta_d^{-1}+o(1/\log N)$.
\end{enumerate}
\end{lemma}

\begin{proof}
(i) Choose one vertex (an anchor) in each clean cluster; every other member
lies in the ball of radius $2(k-1)L_N$ around its anchor, which has
$N^{o(1)}$ vertices.  For each bad cluster choose one vertex in
$\mathcal B_{N,k}$, which has $N^{o(1)}$ choices by
Proposition~\ref{prop:diagonal-event}, and again $N^{o(1)}$ choices for its
remaining members.  The number of good tuples is at least
$(N)_k-k|\mathcal B_{N,k}|N^{k-1}-N^{k-1+o(1)}=(1+o(1))N^k$ by
Proposition~\ref{prop:diagonal-event}.

(ii) By Proposition~\ref{prop:rough-capacity-bound},
$\capGT(A)\ge\sum_i\capGT(A_i)-o(1)$.  For a clean singleton,
Proposition~\ref{prop:local-green-comparison} and
Lemma~\ref{lem:tree-basic-capacities} give
$\capGT(A_i)=\theta_d^{-1}+o(1)$; for a clean non-singleton cluster,
Proposition~\ref{prop:local-green-comparison} and
Lemma~\ref{lem:tree-cluster-gap} give
$\capGT(A_i)\ge\frac{2(d-1)}d\theta_d^{-1}-o(1)$; for a bad cluster,
$\capGT(A_i)\ge\delta_{d,k}$ by Proposition~\ref{prop:rough-capacity-bound}.

(iii) The clusters of a good tuple are clean singletons, so
Proposition~\ref{prop:local-green-comparison} and
Lemma~\ref{lem:tree-basic-capacities} give the claim.
\end{proof}

\section{Late points and the stationary-start cover time}
\label{sec:proof-main}

We now prove Theorem~\ref{thm:rrg-cover-time} for the stationary walk.
The idea is the coupon-collector heuristic made rigorous: we show that the
number of vertices not yet visited at time $\theta_dN(\log N+s)$ is
asymptotically Poisson with mean $e^{-s}$.  Independence of the events
$\{H_v>t\}$ is never assumed.  Instead, we compute falling factorial
moments, which reduces the problem to avoidance probabilities of fixed
finite sets.  These are given by the finite-target theorem
(Section~\ref{sec:fvtl}), the local Green comparison
(Section~\ref{sec:rrg-estimates}) and the tree capacities
(Section~\ref{sec:tree-capacities}).

\subsection{Late points and factorial moments}
\label{sec:late-point-reduction}

Let $G=(V,E)$ be a finite graph with $|V|=N$ and let $(X_t)$ be simple
random walk on $G$ started from $\pi=\pi_G$.  For $s\in\mathbb R$ set
\begin{equation}
   t_N(s):=\big\lfloor\theta_dN(\log N+s)\big\rfloor,
   \qquad
   \mathcal U_N(s):=\{v\in V:H_v>t_N(s)\},
   \qquad
   Y_N(s):=|\mathcal U_N(s)|.
   \label{eq:late-point-count}
\end{equation}
Thus $\mathcal U_N(s)$ is the set of late points and
$\{T_{\rm cov}\le t_N(s)\}=\{Y_N(s)=0\}$.  For $k\ge1$ let
\begin{equation}
   V_{N,\neq}^k:=\{(v_1,\dots,v_k)\in V^k:v_i\ne v_j\text{ for }i\ne j\},
   \qquad
   A(\mathbf v):=\{v_1,\dots,v_k\}
   \quad\text{for }\mathbf v\in V_{N,\neq}^k.
\end{equation}
We write $(y)_k:=y(y-1)\cdots(y-k+1)$ for the falling factorial.

\begin{lemma}
\label{lem:factorial-moment-identity}
For every $k\ge1$ and $s\in\mathbb R$,
\begin{equation}
   \bbE_{\pi}[(Y_N(s))_k]
   =
   \sum_{\mathbf v\in V_{N,\neq}^k}
   \bbP_{\pi}\big(H_{A(\mathbf v)}>t_N(s)\big).
   \label{eq:factorial-moment-identity}
\end{equation}
\end{lemma}

\begin{proof}
$(Y_N(s))_k$ is the number of ordered $k$-tuples of distinct vertices of
$\mathcal U_N(s)$, and $\mathbf v\in V_{N,\neq}^k$ lies in
$\mathcal U_N(s)^k$ if and only if $H_{A(\mathbf v)}>t_N(s)$.  Hence
$(Y_N(s))_k=\sum_{\mathbf v\in V_{N,\neq}^k}\one_{\{H_{A(\mathbf v)}>t_N(s)\}}$,
and taking expectations of this finite sum gives the claim.
\end{proof}

The proof uses only that $H_v$ is the hitting time of $v$ by some process
on $V$; in particular, Lemma~\ref{lem:factorial-moment-identity} applies
verbatim to the lazy walk, to the continuous-time walk and to the union of
several independent walks considered in
Sections~\ref{sec:lazy-continuous}--\ref{sec:late-set-structure}.

Since the Poisson distribution is determined by its moments, convergence
$\bbE_\pi[(Y_N(s))_k]\to e^{-ks}$ for every fixed $k\ge1$ implies
$Y_N(s)\Rightarrow{\rm Poisson}(e^{-s})$.  The rest of this section
verifies this convergence on the typical graph event.

\subsection{Poisson convergence}

\begin{proposition}
\label{prop:poisson-rrg}
Let $G=G_N\in\mathcal E_N$, the diagonal event of
Proposition~\ref{prop:diagonal-event}.  For every fixed $s\in\mathbb R$ and
$k\ge1$,
\begin{equation}
   \bbE_{\pi}[(Y_N(s))_k]\longrightarrow e^{-ks}.
\end{equation}
Consequently $Y_N(s)\Rightarrow{\rm Poisson}(e^{-s})$ under stationary
start, and
\begin{equation}
   \bbP_{\pi}(T_{\rm cov}\le t_N(s))\longrightarrow e^{-e^{-s}}.
\end{equation}
\end{proposition}

In particular, since $\bbP(G_N\in\mathcal E_N)\to1$, this proves the
stationary-start version of Theorem~\ref{thm:rrg-cover-time} in probability
over $G_N$.

\begin{proof}
Fix $k\ge1$ and let $\mathcal G_{N,k}\subseteq V_{N,\neq}^k$ be the set of
good tuples in the sense of Lemma~\ref{lem:cluster-profiles}; by that
lemma, $|\mathcal G_{N,k}|=(1+o(1))N^k$.  By
Lemma~\ref{lem:factorial-moment-identity} it suffices to show that,
uniformly for $\mathbf v\in\mathcal G_{N,k}$,
\begin{equation}
   \bbP_{\pi}\big(H_{A(\mathbf v)}>t_N(s)\big)
   =(1+o(1))N^{-k}e^{-ks},
   \label{eq:good-tuple-asymptotics}
\end{equation}
and that
\begin{equation}
   \sum_{\mathbf v\in V_{N,\neq}^k\setminus\mathcal G_{N,k}}
   \bbP_{\pi}\big(H_{A(\mathbf v)}>t_N(s)\big)
   =o(1).
   \label{eq:bad-tuple-negligible}
\end{equation}

\emph{Good tuples.}  For $\mathbf v\in\mathcal G_{N,k}$,
Lemma~\ref{lem:cluster-profiles}\textup{(iii)} gives
$\capGT(A(\mathbf v))=k\theta_d^{-1}+o(1/\log N)$, so by
Theorem~\ref{thm:uniform-finite-target-fvtl}, uniformly over good tuples,
\begin{align}
   \bbP_\pi(H_{A(\mathbf v)}>t_N(s))
   &=\exp\left\{-\theta_d(\log N+s)(k\theta_d^{-1}+o(1/\log N))+o(1)\right\} \\
   &=(1+o(1))N^{-k}e^{-ks},
\end{align}
which is \eqref{eq:good-tuple-asymptotics}.

\emph{Bad tuples.}  A tuple $\mathbf v\notin\mathcal G_{N,k}$ has a profile
$(a_1,a_2,b)$ with $a_2\ge1$ or $b\ge1$.  By
Lemma~\ref{lem:cluster-profiles}\textup{(ii)} and
Theorem~\ref{thm:uniform-finite-target-fvtl}, each such tuple satisfies
\begin{equation}
   \bbP_\pi\big(H_{A(\mathbf v)}>t_N(s)\big)
   \le e^{-\theta_d(\log N+s)\left(a_1\theta_d^{-1}
      +a_2\frac{2(d-1)}{d}\theta_d^{-1}+b\delta_{d,k}-o(1)\right)}
   =N^{-a_1-a_2\frac{2(d-1)}{d}-b\theta_d\delta_{d,k}+o(1)};
   \label{eq:bad-tuple-bound}
\end{equation}
note that the $o(1)$ error in \eqref{eq:profile-capacity} costs only a
factor $N^{o(1)}$.  Multiplying by the number $N^{a_1+a_2+o(1)}$ of tuples
with this profile (Lemma~\ref{lem:cluster-profiles}\textup{(i)}) and
summing over the finitely many profiles,
\begin{equation}
   \sum_{\mathbf v\in V_{N,\neq}^k\setminus\mathcal G_{N,k}}
   \bbP_{\pi}\big(H_{A(\mathbf v)}>t_N(s)\big)
   \le\sum_{(a_1,a_2,b)}N^{-a_2(d-2)/d-b\theta_d\delta_{d,k}+o(1)}
   \longrightarrow0,
\end{equation}
since every profile in the sum has $a_2\ge1$ or $b\ge1$.  This proves
\eqref{eq:bad-tuple-negligible}.

Combining \eqref{eq:good-tuple-asymptotics}, \eqref{eq:bad-tuple-negligible},
$|\mathcal G_{N,k}|=(1+o(1))N^k$ and
Lemma~\ref{lem:factorial-moment-identity} gives
$\bbE_{\pi}[(Y_N(s))_k]\to e^{-ks}$.  The Poisson limit follows by the
method of moments, and the last assertion follows from
$\{T_{\rm cov}\le t_N(s)\}=\{Y_N(s)=0\}$.
\end{proof}

\section{From stationarity to deterministic starting points}
\label{sec:deterministic-start}

We now remove the assumption of a stationary start.  After
$\ell_N=O(\log N)$ steps the walk is close to stationary, and the vertices
visited before that time are too few to affect the late points.  The same
argument is used for several independent walks
(Section~\ref{sec:parallel-walks}) and for
the late-point sets (Section~\ref{sec:late-set-structure}).  We therefore
formulate it for a general Markov chain whose states occupy vertices.

\paragraph{Setting.}
Let $(Z_t)_{t\ge0}$ be a Markov chain on a finite set $\Omega$ with
transition matrix $P_\Omega$ and stationary distribution $\mu$, and let
$\psi$ assign to each $z\in\Omega$ a set $\psi(z)\subseteq V$ with
$|\psi(z)|\le q$.  For $v\in V$ and $t\ge0$ put
\begin{equation}
   H_v:=\inf\{t\ge0:v\in\psi(Z_t)\},
   \qquad
   R_t:=\bigcup_{j=0}^t\psi(Z_j),
   \qquad
   T_{\rm cov}:=\inf\{t\ge0:R_t=V\},
\end{equation}
and $\eta_\Omega(\ell):=\max_{z\in\Omega}\|P_\Omega^\ell(z,\cdot)-\mu\|_{\TV}$.
Simple random walk is the case $\Omega=V$, $\psi(z)=\{z\}$, $q=1$, $\mu=\pi$;
$q$ independent walks correspond to $\Omega=V^q$,
$\psi(z)=\{z_1,\dots,z_q\}$ and $\mu=\pi^{\otimes q}$.  The test-function
bound \eqref{eq:tv-test-function} holds for any finite Markov chain, with
$\eta_\Omega$ in place of $\eta_N$; its proof in Lemma~\ref{lem:mixing}
uses only the definition of total variation distance.

\begin{proposition}[Comparison with a stationary start]
\label{prop:deterministic-start-comparison}
Let $z\in\Omega$ and let $\ell,\tau\ge0$ be integers.
\begin{enumerate}[label=\textup{(\roman*)}]
\item There is a coupling of the chain $(Z_t)$ started at $z$ with a chain
$(\widetilde Z_t)$ started from $\mu$ such that
$\bbP\bigl(Z_{\ell+t}=\widetilde Z_t\text{ for all }t\ge0\bigr)
\ge1-\eta_\Omega(\ell)$.  Moreover, conditionally on
$(Z_0,\dots,Z_\ell)$ and $\widetilde Z_0$, the process $(\widetilde Z_t)$ is
a Markov chain started at $\widetilde Z_0$.
\item We have
\begin{equation}
   \bigl|\bbP_z(T_{\rm cov}\le\ell+\tau)-\bbP_\mu(T_{\rm cov}\le\tau)\bigr|
   \le\eta_\Omega(\ell)+q(\ell+1)\max_{y\in\Omega,\,v\in V}\bbP_y(H_v>\tau).
   \label{eq:deterministic-start-comparison}
\end{equation}
\end{enumerate}
\end{proposition}

\begin{proof}
(i) Let $(U,\widetilde Z_0)$ be a maximal coupling of $P_\Omega^\ell(z,\cdot)$
and $\mu$, so that $\bbP(U\ne\widetilde Z_0)\le\eta_\Omega(\ell)$.  Given
$U$, sample $(Z_0,\dots,Z_\ell)$ from the law of the chain started at $z$
conditioned on $Z_\ell=U$.  If $U=\widetilde Z_0$, run a single chain
$(Y_t)$ from $U$ and set $Z_{\ell+t}=\widetilde Z_t=Y_t$; otherwise continue
both from their current states with independent chains.  By the Markov
property both processes have the correct laws, and the last assertion
holds by construction.

(ii) Let $S:=R_\ell$, so that $|S|\le q(\ell+1)$, and let $R'_\tau$ be the
range of $(Z_{\ell+t})_{0\le t\le\tau}$.  Then
$\{T_{\rm cov}\le\ell+\tau\}=\{S\cup R'_\tau=V\}$.  Put
$f(y):=\bbP_y(T_{\rm cov}\le\tau)$.  By the Markov property at time $\ell$,
\begin{equation}
   f(Z_\ell)
   \le\bbP_z\bigl(S\cup R'_\tau=V\,\big|\,Z_0,\dots,Z_\ell\bigr)
   \le f(Z_\ell)+\sum_{v\in S}\bbP_{Z_\ell}(H_v>\tau),
\end{equation}
because $S\cup R'_\tau=V$ can hold while $R'_\tau\ne V$ only if the
continuation misses a vertex of $S$.  Taking expectations and using
\eqref{eq:tv-test-function} for $f$, which gives
$|\bbE_z[f(Z_\ell)]-\bbP_\mu(T_{\rm cov}\le\tau)|\le\eta_\Omega(\ell)$,
proves \eqref{eq:deterministic-start-comparison}.
\end{proof}

The comparison shifts the time by $\ell$.  The following lemma shows
that shifts of size $o(N)$ are negligible on the fluctuation scale $N$.

\begin{lemma}[Time shifts]
\label{lem:time-shift}
Let $a>0$, $u_N(s):=\lfloor aN(\log N+s)\rfloor$, and let $F$ be a
continuous distribution function.  Let $\mathcal T_N$ be random variables
(possibly defined on different probability spaces) such that
$\bbP(\mathcal T_N\le u_N(s))\to F(s)$ for every $s\in\mathbb R$.  If
$\ell_N\ge0$ are integers with $\ell_N=o(N)$, then
$\bbP(\mathcal T_N\le u_N(s)-\ell_N)\to F(s)$ for every $s\in\mathbb R$.
\end{lemma}

\begin{proof}
Fix $\varepsilon>0$.  Since $u_N(s)-u_N(s-\varepsilon)\ge aN\varepsilon-1
\ge\ell_N$ for all large $N$, we have
$u_N(s-\varepsilon)\le u_N(s)-\ell_N\le u_N(s)$, hence by monotonicity
\begin{equation}
   F(s-\varepsilon)
   \le\liminf_{N\to\infty}\bbP(\mathcal T_N\le u_N(s)-\ell_N)
   \le\limsup_{N\to\infty}\bbP(\mathcal T_N\le u_N(s)-\ell_N)
   \le F(s).
\end{equation}
Let $\varepsilon\downarrow0$ and use continuity of $F$.
\end{proof}

For simple random walk on $G\in\mathcal E_N$ we take
\begin{equation}
   \ell_N:=\Bigl\lceil\frac{2\log N}{\log(1/\rho)}\Bigr\rceil,
   \qquad\text{so that}\qquad
   \eta_N(\ell_N)\le\tfrac12N^{-3/2}
   \label{eq:ell-N}
\end{equation}
by \eqref{eq:tv-mixing}.  For fixed $s\in\mathbb R$ write
$\tau_N:=t_N(s)-\ell_N$.

\begin{lemma}[Singletons from arbitrary starting points]
\label{lem:arbitrary-start-singletons}
Let $G\in\mathcal E_N$ and fix $s\in\mathbb R$.  There is $c>0$ such that,
for all large $N$,
\begin{equation}
   \max_{y,v\in V_N}\bbP_y(H_v>\tau_N)\le N^{-c}.
\end{equation}
\end{lemma}

\begin{proof}
By the Markov property at time $\ell_N$ and \eqref{eq:tv-test-function},
$\bbP_y(H_v>\tau_N)\le\bbP_\pi(H_v>\tau_N-\ell_N)+\eta_N(\ell_N)$.
Theorem~\ref{thm:uniform-finite-target-fvtl} and
Proposition~\ref{prop:rough-capacity-bound} give
$\bbP_\pi(H_v>\tau_N-\ell_N)\le\exp\{-\delta_{d,1}(\tau_N-\ell_N)/N+o(1)\}
=N^{-\theta_d\delta_{d,1}+o(1)}$.  Together with \eqref{eq:ell-N}, the claim
holds with any $c<\min\{\theta_d\delta_{d,1},3/2\}$.
\end{proof}

\begin{proof}[Proof of Theorem~\ref{thm:rrg-cover-time}, deterministic start]
Let $G\in\mathcal E_N$ and $x\in V_N$.
Proposition~\ref{prop:deterministic-start-comparison}\textup{(ii)} with
$\Omega=V_N$, $q=1$, $\ell=\ell_N$ and $\tau=\tau_N$, together with
\eqref{eq:ell-N} and Lemma~\ref{lem:arbitrary-start-singletons}, gives
\begin{equation}
   \max_{x\in V_N}\bigl|\bbP_x(T_{\rm cov}\le t_N(s))
      -\bbP_\pi(T_{\rm cov}\le t_N(s)-\ell_N)\bigr|
   \le\tfrac12N^{-3/2}+(\ell_N+1)N^{-c}\longrightarrow0.
\end{equation}
By Proposition~\ref{prop:poisson-rrg} and Lemma~\ref{lem:time-shift} with
$a=\theta_d$ and $F(s)=e^{-e^{-s}}$,
$\bbP_\pi(T_{\rm cov}\le t_N(s)-\ell_N)\to e^{-e^{-s}}$.  Hence
\begin{equation}
   \max_{x\in V_N}\left|
   \bbP_x(T_{\rm cov}\le t_N(s))-e^{-e^{-s}}
   \right|\longrightarrow0
\end{equation}
for every $G\in\mathcal E_N$, and since $\bbP(G_N\in\mathcal E_N)\to1$ this
proves Theorem~\ref{thm:rrg-cover-time}.
\end{proof}

\section{Lazy and continuous-time walks}
\label{sec:lazy-continuous}

The cover-time law transfers to lazy and continuous-time walks by an
independent change of clock.  The clock fluctuations are negligible on
the fluctuation scale $N$.

Fix $p\in(0,1]$ and let
\begin{equation}
   P_{G,p}:=(1-p)I+pP_G.
\end{equation}
Thus the walk moves to a uniformly chosen neighbor with probability $p$
and otherwise stays put.  Denote its cover time by $T_{\rm cov}^{(p)}$ and
put
\begin{equation}
   t_{N,p}(s):=
   \left\lfloor\frac{\theta_d}{p}N(\log N+s)\right\rfloor.
\end{equation}
Let $\widehat X$ be the continuous-time walk with generator $P_G-I$
(unit-rate exponential holding times), and let $\widehat T_{\rm cov}$ be
its cover time.

\begin{theorem}[Lazy and continuous-time cover times]
\label{thm:lazy-continuous}
On the same graph events $\mathcal E_N$ as in
Proposition~\ref{prop:diagonal-event}, for every fixed $p\in(0,1]$ and
$s\in\mathbb R$,
\begin{equation}
   \sup_{x\in V_N}\left|
   \bbP_x^{G,p}\bigl(T_{\rm cov}^{(p)}\le t_{N,p}(s)\bigr)
   -e^{-e^{-s}}\right|\longrightarrow0.
\end{equation}
Moreover,
\begin{equation}
   \sup_{x\in V_N}\left|
   \bbP_x^G\bigl(\widehat T_{\rm cov}
      \le\theta_dN(\log N+s)\bigr)
   -e^{-e^{-s}}\right|\longrightarrow0.
\end{equation}
Equivalently, the two centered and scaled cover times converge to a standard
Gumbel random variable, with scales $\theta_dN/p$ and $\theta_dN$,
respectively.
\end{theorem}

\begin{lemma}[Clock concentration]
\label{lem:clock-concentration}
Let $W_1,W_2,\ldots$ be independent identically distributed nonnegative
random variables with mean $a$ and finite variance $\sigma^2$, whose common
law does not depend on $N$.  Write $S_0:=0$ and
$S_m:=\sum_{j=1}^mW_j$ for $m\ge1$.  Let $T_N$ be a nonnegative
integer-valued random variable independent of the $W_j$.  If
$T_N=O_{\bbP}(N\log N)$, uniformly over a family of laws for $T_N$, then
\begin{equation}
   S_{T_N}-aT_N=o_{\bbP}(N)
\end{equation}
uniformly over that family.
\end{lemma}

\begin{proof}
For fixed $M,\varepsilon>0$, set $L_N':=\lceil MN\log N\rceil$.
Independence and Chebyshev's inequality give, uniformly over the family,
\begin{align}
   \bbP\bigl(|S_{T_N}-aT_N|>\varepsilon N\bigr)
   &\le \bbP(T_N>L_N')
      +\sup_{0\le m\le L_N'}\bbP(|S_m-am|>\varepsilon N)\\
   &\le \bbP(T_N>L_N')+\frac{\sigma^2L_N'}{\varepsilon^2N^2}.
\end{align}
First let $N\to\infty$, and then let $M\to\infty$, using uniform
tightness of $T_N/(N\log N)$.
\end{proof}

\begin{proof}[Proof of Theorem~\ref{thm:lazy-continuous}]
Work on $G\in\mathcal E_N$.  Theorem~\ref{thm:rrg-cover-time} implies
$T_{\rm cov}=O_{\bbP}(N\log N)$ uniformly in the starting vertex.
Indeed, for any $M>\theta_d$ and fixed $R$, eventually
$MN\log N\ge\theta_dN(\log N+R)$, so the uniform Gumbel law gives
\begin{equation}
   \limsup_{N\to\infty}\sup_x
   \bbP_x^G(T_{\rm cov}>MN\log N)
   \le 1-e^{-e^{-R}}.
\end{equation}
Letting $R\to\infty$ proves the required bound.

Construct the lazy walk from its simple-random-walk jump chain by inserting
independent geometric waiting times $W_j$ on $\{1,2,\ldots\}$ with success
probability $p$.  These clocks have mean $a=1/p$ and variance
$(1-p)/p^2$, and
\begin{equation}
   T_{\rm cov}^{(p)}=S_{T_{\rm cov}}.
\end{equation}
For $p=1$ the waiting times equal one.  Similarly, independent rate-one
exponential waiting times, with mean $a=1$ and variance one, give
\begin{equation}
   \widehat T_{\rm cov}=S_{T_{\rm cov}}.
\end{equation}
In either construction, write $\mathcal C_N:=S_{T_{\rm cov}}$ for the
resulting cover time.  Lemma~\ref{lem:clock-concentration} gives
$\mathcal C_N-aT_{\rm cov}=o_{\bbP}(N)$ uniformly over starting vertices.
Thus for every fixed $\varepsilon>0$ and $s\in\mathbb R$,
\begin{align}
   \bbP_x^G\bigl(T_{\rm cov}\le
       \theta_dN(\log N+s-\varepsilon)\bigr)-o(1)
   &\le\bbP_x\bigl(\mathcal C_N\le
       a\theta_dN(\log N+s)\bigr)\\
   &\le\bbP_x^G\bigl(T_{\rm cov}\le
       \theta_dN(\log N+s+\varepsilon)\bigr)+o(1),
\end{align}
where the errors are uniform in $x$.  Apply
Theorem~\ref{thm:rrg-cover-time} at $s-\varepsilon$ and $s+\varepsilon$,
then let $\varepsilon\downarrow0$.  Continuity of the Gumbel distribution
proves both assertions.  For the lazy walk the real threshold may be
replaced by its integer part because $T_{\rm cov}^{(p)}$ is integer-valued.
\end{proof}

\section{A fixed number of independent parallel walks}
\label{sec:parallel-walks}

Fix an integer $q\ge1$.  Let
$X^{(1)},\ldots,X^{(q)}$ be independent discrete-time simple random walks on
the same graph, all taking one step per round, and define
\begin{equation}
   T_{\rm cov}^{[q]}:=\inf\left\{t\ge0:
      \bigcup_{i=1}^q\{X_0^{(i)},\ldots,X_t^{(i)}\}=V_N\right\}.
\end{equation}
Set
\begin{equation}
   t_{N,q}(s):=
   \left\lfloor\frac{\theta_d}{q}N(\log N+s)\right\rfloor.
\end{equation}
The symbol $q$ here is the number of walkers and is unrelated to the
constants $q_{d,k}$ of Section~\ref{sec:late-set-structure}.

\begin{theorem}[Second-order law for $q$ parallel walks]
\label{thm:parallel-walks}
For every fixed integer $q\ge1$, on the same graph events $\mathcal E_N$ as
above, and for every fixed $s\in\mathbb R$,
\begin{equation}
   \sup_{(x_1,\ldots,x_q)\in V_N^q}\left|
   \bbP_{x_1,\ldots,x_q}^{G}
      \bigl(T_{\rm cov}^{[q]}\le t_{N,q}(s)\bigr)
   -e^{-e^{-s}}\right|\longrightarrow0.
\end{equation}
Equivalently,
\begin{equation}
   \frac{T_{\rm cov}^{[q]}-(\theta_d/q)N\log N}
        {(\theta_d/q)N}
   \Rightarrow\Lambda.
\end{equation}
This refines, at fixed $q$, the first-order $q$-fold speed-up established in
\cite{CooperFriezeRadzik2010}.
\end{theorem}

\begin{proof}[Proof from stationary starting points]
Let $Y_N^{[q]}(s)$ be the number of vertices unvisited by every one of the
$q$ walks at time $t_{N,q}(s)$, and let $H_A^{(i)}$ denote the hitting time
of $A$ by walker $i$.  If all walkers start independently from $\pi$, then
for every finite target $A$, independence gives the exact identity
\begin{equation}
   \bbP_{\pi^{\otimes q}}^G
      \bigl(H_A^{(i)}>t\text{ for every }1\le i\le q\bigr)
   =\bigl[\bbP_\pi^G(H_A>t)\bigr]^q.
\end{equation}
Consequently, for every fixed $k$,
\begin{equation}
   \bbE_{\pi^{\otimes q}}^G\bigl[(Y_N^{[q]}(s))_k\bigr]
   =\sum_{\mathbf v\in V_{N,\neq}^k}
      \bigl[\bbP_\pi^G(H_{A(\mathbf v)}>t_{N,q}(s))\bigr]^q.
\end{equation}
For a good tuple, the finite-target theorem and
$\capGT(A)=k\theta_d^{-1}+o(1/\log N)$ yield
\begin{align}
   \bigl[\bbP_\pi^G(H_A>t_{N,q}(s))\bigr]^q
   &=\exp\left\{-q\frac{t_{N,q}(s)}N
        \bigl(k\theta_d^{-1}+o(1/\log N)\bigr)+o(1)\right\}\\
   &=(1+o(1))N^{-k}e^{-ks}.
\end{align}
The same good set $\mathcal G_{N,k}$ as in
Proposition~\ref{prop:poisson-rrg} has $(1+o(1))N^k$ elements, so its total
contribution tends to $e^{-ks}$.
For a bad-tuple profile $(a_1,a_2,b)$, multiplication of the exponent by
$q$ exactly cancels the factor $1/q$ in $t_{N,q}(s)$.  Thus, by Lemma~\ref{lem:cluster-profiles}, its total
contribution is bounded, exactly as in
Proposition~\ref{prop:poisson-rrg}, by
\begin{equation}
   N^{-a_2(d-2)/d-b\theta_d\delta_{d,k}+o(1)}=o(1).
\end{equation}
There are only finitely many profiles for fixed $k$, so their total
contribution is $o(1)$.
The factorial moments converge to $(e^{-s})^k$, so
$Y_N^{[q]}(s)\Rightarrow{\rm Poisson}(e^{-s})$.  Since
\begin{equation}
   \{T_{\rm cov}^{[q]}\le t_{N,q}(s)\}
   =\{Y_N^{[q]}(s)=0\}=\{Y_N^{[q]}(s)<1/2\},
\end{equation}
and $1/2$ is a continuity point of the limiting distribution, the
stationary-start cover-time law follows.
\end{proof}

\begin{proof}[Passage to arbitrary starting $q$-tuples]
We apply Proposition~\ref{prop:deterministic-start-comparison} to the
product chain on $\Omega=V_N^q$ with $\mu=\pi^{\otimes q}$ and
$\psi(z)=\{z_1,\dots,z_q\}$.  Since the total-variation distance between
product measures is at most the sum of the distances between their
marginals, $\eta_\Omega(\ell)\le q\eta_N(\ell)$.  Choose
$\ell_N=\lceil C_q\log N\rceil$ so large that $\eta_N(\ell_N)\le N^{-10}$,
and put $\tau_N=t_{N,q}(s)-\ell_N$.  With
$m_N=\lceil C_q'\log N\rceil$ such that $\eta_N(m_N)\le N^{-10}$, the
Markov property, \eqref{eq:tv-test-function}, the singleton case of the
finite-target theorem and the rough capacity bound give
\begin{align}
   \sup_{y,v\in V_N}\bbP_y^G(H_v>\tau_N)
   &\le \sup_{v\in V_N}\bbP_\pi^G(H_v>\tau_N-m_N)+N^{-10}\\
   &\le \exp\left\{-\left[\frac{\theta_d}{q}\log N+O(1)\right]
          \delta_{d,1}+o(1)\right\}+N^{-10}
   \le N^{-\gamma}
\end{align}
for all large $N$ and any fixed
$0<\gamma<\min\{10,\theta_d\delta_{d,1}/q\}$.  For the product chain, the
first time some walker visits $v$ is at most the hitting time of $v$ by the
first walker, so $\max_{y\in\Omega}\bbP_y(H_v>\tau_N)\le N^{-\gamma}$ as
well.  Proposition~\ref{prop:deterministic-start-comparison}\textup{(ii)}
gives, uniformly over starting $q$-tuples,
\begin{equation}
   \left|\bbP_{x_1,\ldots,x_q}^G
      (T_{\rm cov}^{[q]}\le t_{N,q}(s))
   -\bbP_{\pi^{\otimes q}}^G
      (T_{\rm cov}^{[q]}\le\tau_N)\right|
   \le qN^{-10}+q(\ell_N+1)N^{-\gamma}=o(1).
\end{equation}
Lemma~\ref{lem:time-shift} with $a=\theta_d/q$ and the stationary result
complete the proof.
\end{proof}

\section{Spatial distribution and small components of late points}
\label{sec:late-set-structure}

We now study the unvisited set at times before and within the cover-time
window, in the model of Theorem~\ref{thm:rrg-cover-time}.
The results are stated first, followed by the additional estimates and
their proofs.  We retain the notation of the preceding sections.
All parameters are fixed as $N\to\infty$.  Limits are conditional on
$G_N$, in probability over $G_N$, and uniform over starting vertices;
$L^2$ limits refer to stationary start.  Weights, vertex sets and marks
are measurable functions of $G_N$, chosen without observing the walk.

\paragraph{Results.}

For $s\in\mathbb R$ and $0<\gamma\le1$, set
\begin{equation}
 t_N^{(\gamma)}(s):=\lfloor\theta_dN(\gamma\log N+s)\rfloor,
 \qquad
 \mathcal U_N^{(\gamma)}(s):=\{v:H_v>t_N^{(\gamma)}(s)\}.
 \label{eq:late-time-set-definition}
\end{equation}
For $0<\gamma<1$ and $f:V_N\to[-1,1]$, put
\[
 R_N(f):=N^{\gamma-1}\sum_{v\in\mathcal U_N^{(\gamma)}(s)}f(v)
              -\frac{e^{-s}}N\sum_{v\in V_N}f(v).
\]
Let $A_N\subseteq V_N$.  A ratio with denominator
$|\mathcal U_N^{(\gamma)}(s)|=0$ is assigned value zero.
For a fixed compact metric space $E$ and maps $\xi_N:V_N\to E$, define
\[
 \Xi_N(s):=\sum_{v\in\mathcal U_N^{(1)}(s)}\delta_{\xi_N(v)}.
\]
Let $\nu$ be a fixed Borel probability measure on $E$.
Write $\mathcal M_p(E)$ for the finite point measures with the weak
topology and $\operatorname{PPP}(\mu)$ for the Poisson random measure
with intensity $\mu$.

\begin{theorem}[Size and spatial distribution of unvisited vertices]
\label{thm:late-set-equidistribution}
\leavevmode
\begin{enumerate}[label=\textup{(\roman*)}]
\item For $0<\gamma<1$ and $\varepsilon>0$,
\begin{equation}
 \sup_f\bbE_{\pi_G}^G[R_N(f)^2]\longrightarrow0,
 \qquad
 \sup_{x\in V_N}\sup_f\bbP_x^G(|R_N(f)|>\varepsilon)
          \longrightarrow0.
 \label{eq:late-weighted-limits}
\end{equation}
Consequently, in probability,
\begin{equation}
 N^{\gamma-1}|\mathcal U_N^{(\gamma)}(s)|\longrightarrow e^{-s},
 \qquad
 \frac{|\mathcal U_N^{(\gamma)}(s)\cap A_N|}
      {|\mathcal U_N^{(\gamma)}(s)|}-\frac{|A_N|}{N}\longrightarrow0.
 \label{eq:late-size-equidistribution}
\end{equation}
\item If $N^{-1}\sum_v\delta_{\xi_N(v)}\Rightarrow\nu$ weakly in
probability over $G_N$, then
\begin{equation}
 \Xi_N(s)\Longrightarrow\operatorname{PPP}(e^{-s}\nu)
 \quad\text{in }\mathcal M_p(E).
 \label{eq:critical-late-PPP}
\end{equation}
\end{enumerate}
\end{theorem}

Thus $\xi_N(v)=v/N$ on $V_N=\{1,\ldots,N\}$ gives intensity
$e^{-s}\,dx$ on $[0,1]$.
For any fixed number of pairwise disjoint sets $A_{N,j}$ with
$|A_{N,j}|/N\to a_j$ in probability over $G_N$ for constants $a_j$, the counts
$|\mathcal U_N^{(1)}(s)\cap A_{N,j}|$ converge jointly to independent
$\operatorname{Poisson}(e^{-s}a_j)$ variables.

Let $\mathfrak T_d(k)$ be the unlabelled trees with $k\ge1$ vertices and
maximum degree at most $d$.  For $T\in\mathfrak T_d(k)$, set
\begin{equation}
 q_{d,k}:=\frac{(d-2)k+2}{d},\qquad
 \gamma_{d,k}:=q_{d,k}^{-1},\qquad
 \kappa_{d,T}:=
 \frac{\prod_{u\in V(T)}(d)_{\deg_T(u)}}
      {|\operatorname{Aut}(T)|d^{k-1}}.
 \label{eq:late-component-thresholds}
\end{equation}
In particular $q_{d,1}=1$ and
$q_{d,2}=2(d-1)/d$.  Let $\mathcal S_N(T):=\{A\subseteq V_N:G_N[A]\simeq T\}$.
Write $Z_{T,N}^{(\gamma)}(s)$ for the number of sets in
$\mathcal S_N(T)$ contained in $\mathcal U_N^{(\gamma)}(s)$, and
$C_{T,N}^{(\gamma)}(s)$ for the number of connected components of
$G_N[\mathcal U_N^{(\gamma)}(s)]$ isomorphic to $T$.  Put
\begin{equation}
 m_{T,N}(\gamma,s):=\kappa_{d,T}e^{-q_{d,k}s}N^{1-\gamma q_{d,k}},
 \qquad
 \lambda_T(s):=\kappa_{d,T}e^{-q_{d,k}s},
 \label{eq:late-tree-mean-scale}
\end{equation}
and let $(\mathsf P_T:T\in\mathfrak T_d(k))$ be independent Poisson
variables with means $\lambda_T(s)$.

\begin{theorem}[Counts of fixed tree components]
\label{thm:late-component-hierarchy}
For each $k\ge1$ and $T\in\mathfrak T_d(k)$:
\begin{enumerate}[label=\textup{(\roman*)}]
\item If $0<\gamma<\gamma_{d,k}$, then
$Z_{T,N}^{(\gamma)}(s)/m_{T,N}(\gamma,s)\to1$ in $L^2$. Moreover,
 $Z_{T,N}^{(\gamma)}(s)/m_{T,N}(\gamma,s)$ and
$C_{T,N}^{(\gamma)}(s)/m_{T,N}(\gamma,s)$ converge to one in probability.
\item At $\gamma=\gamma_{d,k}$,
\begin{equation}
 \bigl(C_{T,N}^{(\gamma_{d,k})}(s):T\in\mathfrak T_d(k)\bigr)
 \Longrightarrow(\mathsf P_T:T\in\mathfrak T_d(k)).
 \label{eq:late-tree-copy-joint-Poisson}
\end{equation}
The same limit holds for $Z$, and the two vectors coincide with
probability tending to one.
\item If $\gamma_{d,k}<\gamma\le1$, then
$\bbP_x^G(Z_{T,N}^{(\gamma)}(s)>0)\to0$, and hence also
$\bbP_x^G(C_{T,N}^{(\gamma)}(s)>0)\to0$.
\end{enumerate}
\end{theorem}

Using $L_N$ and the $2L_N$-clusters from
Section~\ref{sec:rrg-estimates}, write
$D_N^{(\gamma)}(s):=|E(G_N[\mathcal U_N^{(\gamma)}(s)])|$ and set
\[
 \Lambda_N(\gamma,s):=\frac d2e^{-2(d-1)s/d}N^{1-2\gamma(d-1)/d},
 \quad
 \lambda_d(s):=\frac d2e^{-2(d-1)s/d}.
\]

\begin{corollary}[Adjacent unvisited vertices]
\label{cor:late-independence-transition}
For $0<\gamma<\gamma_{d,2}$, $D_N^{(\gamma)}(s)/\Lambda_N(\gamma,s)\to1$
in $L^2$ and in probability.  At $\gamma=\gamma_{d,2}$,
$D_N^{(\gamma)}(s)\Rightarrow\operatorname{Poisson}(\lambda_d(s))$;
with probability tending to one, every $2L_N$-cluster of the late set is a singleton or
the endpoints of an edge.  For $\gamma_{d,2}<\gamma\le1$, with probability
tending to one no two late vertices are within distance $2L_N$.
Consequently,
\begin{equation}
 \bbP_x^G\bigl(\mathcal U_N^{(\gamma)}(s)\text{ is independent}\bigr)
 \longrightarrow
 \begin{cases}
 0,&\gamma<\gamma_{d,2},\\
 e^{-\lambda_d(s)},&\gamma=\gamma_{d,2},\\
 1,&\gamma>\gamma_{d,2}.
 \end{cases}
 \label{eq:late-independent-transition}
\end{equation}
\end{corollary}

Let $\mathcal M_N^{(\gamma)}(s)$ be the maximum number of vertices in a
component of $G_N[\mathcal U_N^{(\gamma)}(s)]$, with value zero for the
empty set, and put
\begin{equation}
 K_{d,k}:=\sum_{T\in\mathfrak T_d(k)}\kappa_{d,T}
 =\frac{d}{k((d-2)k+2)}\binom{(d-1)k}{k-1}.
 \label{eq:total-tree-copy-density}
\end{equation}
Set $\lambda_{d,k}(s):=K_{d,k}e^{-q_{d,k}s}$.  The identity for
$K_{d,k}$ is proved in Lemma~\ref{lem:connected-tree-capacity}.

\begin{corollary}[Largest component of the unvisited set]
\label{cor:late-largest-component}
For $k\ge1$ and $\gamma_{d,k+1}<\gamma<\gamma_{d,k}$,
$\bbP_x^G(\mathcal M_N^{(\gamma)}(s)=k)\to1$.
At $\gamma=\gamma_{d,k}$,
\begin{equation}
 \sum_{T\in\mathfrak T_d(k)}C_{T,N}^{(\gamma_{d,k})}(s)
 \Longrightarrow\operatorname{Poisson}(\lambda_{d,k}(s));
 \label{eq:critical-size-k-component-count}
\end{equation}
with probability tending to one, every $k$-vertex late component is a
tree and no component has more than $k$ vertices.  Moreover,
\begin{equation}
 \begin{aligned}
 \bbP_x^G(\mathcal M_N^{(\gamma_{d,k})}(s)=k)
       &\longrightarrow1-e^{-\lambda_{d,k}(s)},\\
 \bbP_x^G(\mathcal M_N^{(\gamma_{d,k})}(s)=k-1)
       &\longrightarrow e^{-\lambda_{d,k}(s)}.
 \end{aligned}
 \label{eq:critical-largest-component-law}
\end{equation}
\end{corollary}

Cooper and Frieze~\cite[Theorem~5]{CooperFrieze2013} studied tree-component
counts in the unvisited set.  At the times $t_N^{(\gamma)}(s)$ considered
here, summing \eqref{eq:late-tree-mean-scale} over tree shapes gives the
same asymptotic scale as their formula, and the thresholds $\gamma_{d,k}$
correspond to the same disappearance times.  Their concentration theorem
leaves out the windows in which the expected component counts remain
bounded.  Theorem~\ref{thm:late-component-hierarchy} gives joint independent
Poisson limits for individual tree shapes in these windows.  Our limits
are conditional on the graph, hold in probability over $G_N$, and are
uniform over starting vertices.

\paragraph{Additional estimates.}

The finite-target estimate and local Green-matrix comparison also apply
at the times $t_N^{(\gamma)}(s)$.  We additionally use a coupling of the
entire unvisited sets, the capacity of a connected set in $\Td$ and a
lower bound on its increase when another vertex is added, and the number
of copies of each fixed tree in $G_N$.  These estimates give the weighted
second moments and mixed factorial moments used below.

For the rest of this section, fix a deterministic graph sequence
$G_N\in\mathcal E_N$, with $\mathcal E_N$ from
Proposition~\ref{prop:diagonal-event}.

With $\eta_N$ as in Lemma~\ref{lem:mixing}, choose
$b_N=\lceil C\log N\rceil$ so that $\eta_N(b_N)\le N^{-10}$.
For a stationary walk $\widetilde X$, write
$\widetilde{\mathcal U}_N^{(\gamma)}(s)$ for its unvisited set at time
$t_N^{(\gamma)}(s)$.  Below, $\bbP$ denotes the coupling law conditional
on $G$.

\begin{lemma}
\label{lem:late-set-start-transfer}
For every fixed $0<\gamma\le1$, walks from $x\in V_N$ can be coupled
with stationary walks so that, for every fixed $M>0$,
\begin{equation}
 \sup_{x\in V_N}\sup_{|s|\le M}\bbP\!\left(
  \mathcal U_N^{(\gamma)}(s)\ne
  \widetilde{\mathcal U}_N^{(\gamma)}
       \!\left(s-\frac{b_N}{\theta_dN}\right)
 \right)\longrightarrow0.
 \label{eq:late-set-start-coupling}
\end{equation}
\end{lemma}

\begin{proof}
Put $\tau=t_N^{(\gamma)}(s)-b_N$.  The Markov property after $b_N$
steps, Theorem~\ref{thm:uniform-finite-target-fvtl}, and
Proposition~\ref{prop:rough-capacity-bound} give
\[
 \sup_{y,v}\bbP_y^G(H_v>\tau)
 \le\eta_N(b_N)+\sup_v\bbP_{\pi_G}^G(H_v>\tau-b_N)
 \le N^{-10}+N^{-\gamma\theta_d\delta_{d,1}+o(1)},
\]
uniformly for $|s|\le M$.  Use the coupling of
Proposition~\ref{prop:deterministic-start-comparison}\textup{(i)} with
$\Omega=V_N$ and $\ell=b_N$, and let $S=\{X_0,\ldots,X_{b_N}\}$.
On coupling success the late sets can differ only at vertices of $S$
missed by the stationary continuation.  Conditional on the initial
path and the coupled states, that continuation is a fresh walk from
its initial state.  Thus the probability of unequal late sets is at most
\[
 \eta_N(b_N)+(b_N+1)\sup_{y,v}\bbP_y^G(H_v>\tau)=o(1).
\]
Finally, $\tau=t_N^{(\gamma)}(s-b_N/(\theta_dN))$ exactly.
\end{proof}

Put $\varepsilon_d=((d-2)/d)^2$.

\begin{lemma}
\label{lem:connected-tree-capacity}
For every connected $k$-vertex set $A\subset\Td$ and $z\notin A$,
\begin{equation}
 \Cap_{\Td}(A)=q_{d,k}\theta_d^{-1},\qquad
 \Cap_{\Td}(A\cup\{z\})\ge(q_{d,k}+\varepsilon_d)\theta_d^{-1}.
 \label{eq:connected-tree-capacity-gap}
\end{equation}
For every fixed $T\in\mathfrak T_d(k)$ and every deterministic sequence
$G_N\in\mathcal E_N$,
\begin{equation}
 |\mathcal S_N(T)|=(\kappa_{d,T}+o(1))N.
 \label{eq:tree-copy-enumeration}
\end{equation}
The identity in \eqref{eq:total-tree-copy-density} also holds.
\end{lemma}

\begin{proof}
Lemma~\ref{lem:tree-cluster-gap} gives
$\Cap_{\Td}(A)=\sum_{x\in A}e_A(x)$.  After a first step from $x\in A$
to a vertex outside $A$, the walk can next enter $A$ only at $x$, because
$A$ is connected and the ambient graph is a tree.  Thus
$e_A(x)=\theta_d^{-1}(d-\deg_A(x))/d$.  Summing over $x\in A$ and using
$\sum_x\deg_A(x)=2(k-1)$ proves the first identity.

Let $G_A=G_{\Td}(A)$, $g=(g_{\Td}(a,z):a\in A)$ and
$h_A(z)=\bbP_z(H_A<\infty)=g^\top G_A^{-1}\one_A$.  Block inversion gives
\begin{equation}
 \Cap_{\Td}(A\cup\{z\})-\Cap_{\Td}(A)
 =\frac{(1-h_A(z))^2}{S},\qquad
 S:=\theta_d-g^\top G_A^{-1}g\in(0,\theta_d].
 \label{eq:connected-tree-extension-schur}
\end{equation}
Positivity of $S$ follows from positive definiteness of the tree Green
matrix.  All of $A$ lies in one component of $\Td\setminus\{z\}$.
Taking a first step into any of the other $d-1$ components and never
returning to $z$ gives
$1-h_A(z)\ge(d-1)\theta_d^{-1}/d=(d-2)/d$.  This proves the second bound in
\eqref{eq:connected-tree-capacity-gap}.

Root $T$ at $r$.  For a prescribed root image with a tree neighbourhood,
the number of embeddings of $T$ is
$(d)_{\deg_T(r)}\prod_{u\ne r}(d-1)_{\deg_T(u)-1}$.
Summing over root images and dividing by $|\operatorname{Aut}(T)|$
gives $\kappa_{d,T}N$ up to an error of $N^{o(1)}$: only $N^{o(1)}$
vertices lack the required fixed-radius tree neighbourhood, and each
vertex belongs to $O_{d,k}(1)$ copies of $T$.

For the sum over shapes, let $B(z)$ be the generating function for
finite connected sets containing the root in the infinite rooted tree
with $d-1$ children at each vertex, counted by their number of vertices.
Then $B=z(1+B)^{d-1}$ and
$kK_{d,k}=[z^k]z(1+B)^d$.  Lagrange inversion gives, for $k\ge2$,
\[
 kK_{d,k}=\frac{d}{k-1}\binom{(d-1)k}{k-2}
 =\frac{d}{(d-2)k+2}\binom{(d-1)k}{k-1};
\]
the case $k=1$ is immediate.
\end{proof}

\paragraph{Proofs of the results.}

We first work under stationary start.  All estimates are uniform for
$s$ in compact sets, so Lemma~\ref{lem:late-set-start-transfer} applies
after the time shift $s-b_N/(\theta_dN)$.  For the joint limits, we
expand mixed factorial moments over marked vertices or prescribed tree
copies.  The capacity increase in
Lemma~\ref{lem:connected-tree-capacity} bounds the contribution of
overlapping or nearby copies.

\begin{proof}[Proof of Theorem~\ref{thm:late-set-equidistribution}]
Write $\mathcal C=V_N\setminus\mathcal B_{N,2}$ and
$I_v=\one_{\{v\in\mathcal U_N^{(\gamma)}(s)\}}$.
Theorem~\ref{thm:uniform-finite-target-fvtl} and
Proposition~\ref{prop:local-green-comparison}, together with
$\Cap_{\Td}(\{o\})=\theta_d^{-1}$, give uniformly for $v,w\in\mathcal C$
\begin{equation*}
\begin{aligned}
 p_v:=\bbE_{\pi_G}^G I_v
       &=(1+o(1))e^{-s}N^{-\gamma},\\
 p_{v,w}:=\bbE_{\pi_G}^G[I_vI_w]
       &=(1+o(1))e^{-2s}N^{-2\gamma}
          &&\text{if }\dist_G(v,w)>2L_N,\\
 p_{v,w}&\le N^{-2\gamma(d-1)/d+o(1)}
          &&\text{if }0<\dist_G(v,w)\le2L_N.
\end{aligned}
\end{equation*}
The last bound uses Lemma~\ref{lem:tree-cluster-gap}.
There are at most $N^{1+o(1)}$ ordered pairs of distinct vertices at
distance at most $2L_N$.  After multiplication by $N^{2\gamma-2}$,
the diagonal terms, the terms with distance greater than $2L_N$, and the
remaining terms in the variance of $\sum_{v\in\mathcal C}f(v)I_v$
are bounded in absolute value, uniformly in $|f|\le1$, by
\[
 O(N^{\gamma-1}),\qquad o(1),\qquad
 N^{-1+2\gamma/d+o(1)}+N^{-1+o(1)}=o(1),
\]
respectively.  For the third bound use
$|p_{v,w}-p_vp_w|\le p_{v,w}+p_vp_w$ and $2-2(d-1)/d=2/d$.
The singleton estimate also gives, uniformly in $f$,
\[
 N^{\gamma-1}\sum_{v\in\mathcal C}f(v)p_v
       -\frac{e^{-s}}N\sum_{v\in\mathcal C}f(v)=o(1).
\]
The omitted vertices and their centering contribute at most
$(N^{\gamma-1}+e^{-s}/N)|\mathcal B_{N,2}|=o(1)$ deterministically.
This proves the stationary $L^2$ assertion in
\eqref{eq:late-weighted-limits}.  The choices $f\equiv1$ and
$f=\one_{A_N}$ then give \eqref{eq:late-size-equidistribution}.

For the point measure at $\gamma=1$, first suppose that
$N^{-1}\sum_v\delta_{\xi_N(v)}\Rightarrow\nu$ along the fixed graph
sequence.  Take pairwise disjoint Borel sets $B_1,\ldots,B_m$ with
$\nu(\partial B_j)=0$, and integers $r_j\ge0$ with $r=\sum_jr_j\ge1$.
Expand $\bbE_{\pi_G}^G\prod_j(\Xi_N(s)(B_j))_{r_j}$ over distinct
ordered vertex tuples with the prescribed marks.  As in the proof of
Proposition~\ref{prop:poisson-rrg}, tuples whose vertices avoid
$\mathcal B_{N,r}$ and have pairwise distances greater than $2L_N$
have avoidance probability $(1+o(1))N^{-r}e^{-rs}$.  The total
contribution of all other tuples is $o(1)$ even without the mark
restrictions.  The number of tuples in the first class satisfying the
mark restrictions, divided by $N^r$, tends to
$\prod_j\nu(B_j)^{r_j}$.  Consequently
\[
 \bbE_{\pi_G}^G\prod_j(\Xi_N(s)(B_j))_{r_j}
       \longrightarrow\prod_j(e^{-s}\nu(B_j))^{r_j}.
\]
These mixed factorial moments give independent Poisson counts.
Proposition~\ref{prop:poisson-rrg} gives tightness of the total mass;
compactness of $E$ then gives tightness in $\mathcal M_p(E)$.  The counts
on continuity sets identify the limit in \eqref{eq:critical-late-PPP}.
Restricting coordinates directly to $A_{N,j}$ proves the assertion about
disjoint vertex sets along any graph sequence with
$|A_{N,j}|/N\to a_j$.  Lemma~\ref{lem:late-set-start-transfer} gives all probability and
distributional limits uniformly over starts, including uniformity over $f$ in
\eqref{eq:late-weighted-limits}.  The centering changes by $o(1)$,
uniformly over $f$, when $s$ is replaced by $s-b_N/(\theta_dN)$.
For graph-dependent marks and subsets, every subsequence has a further
subsequence along which $G_N\in\mathcal E_N$ eventually and the assumed
empirical limits hold almost surely.  Applying the preceding
deterministic-sequence argument to each such realization proves the
asserted convergence in probability over $G_N$.
\end{proof}

\begin{proof}[Proof of Theorem~\ref{thm:late-component-hierarchy}]
Write $Z_{T,N}=Z_{T,N}^{(\gamma)}(s)$,
$C_{T,N}=C_{T,N}^{(\gamma)}(s)$ and $m_{T,N}=m_{T,N}(\gamma,s)$, and put
$\mathcal S_N^0(T)=\{A\in\mathcal S_N(T):
A\cap\mathcal B_{N,2k}=\varnothing\}$.
This class contains $(\kappa_{d,T}+o(1))N$ copies, while
$|\mathcal S_N(T)\setminus\mathcal S_N^0(T)|=N^{o(1)}$.
Proposition~\ref{prop:local-green-comparison},
Lemma~\ref{lem:connected-tree-capacity}, and
Theorem~\ref{thm:uniform-finite-target-fvtl} give, uniformly over
$A\in\mathcal S_N^0(T)$,
\begin{equation}
 \bbP_{\pi_G}^G(A\subseteq\mathcal U_N^{(\gamma)}(s))
 =(1+o(1))N^{-\gamma q_{d,k}}e^{-q_{d,k}s}.
 \label{eq:one-late-tree-copy}
\end{equation}
If $\gamma q_{d,k}<1$, copies outside this class contribute at most
$N^{o(1)}=o(m_{T,N})$ deterministically.  There are
$(1+o(1))(\kappa_{d,T}N)^2$ ordered pairs in $\mathcal S_N^0(T)^2$
whose mutual distance exceeds $2L_N$.  Their total contribution to the
second factorial moment is $(1+o(1))m_{T,N}^2$.
The remaining distinct pairs in $\mathcal S_N^0(T)^2$ number at most
$N^{1+o(1)}$: choose the first copy in $O(N)$ ways, a vertex of the second
within distance $2L_N$ of it in $N^{o(1)}$ ways, and a copy through that
vertex in $O_{d,k}(1)$ ways.  Each union is a single $2L_N$-cluster
disjoint from $\mathcal B_{N,2k}$ and contains a connected $k$-set and
an additional vertex.  Its capacity is at least
$(q_{d,k}+\varepsilon_d)\theta_d^{-1}+o(1/\log N)$, so these pairs contribute at most
\begin{equation}
 N^{1-\gamma(q_{d,k}+\varepsilon_d)+o(1)}=o(m_{T,N}^2),
 \label{eq:late-tree-nonfar-second-moment}
\end{equation}
since $-1+\gamma(q_{d,k}-\varepsilon_d)<0$.
The diagonal contributes $O(m_{T,N})=o(m_{T,N}^2)$.
Together with the mean estimate this proves the $L^2$ assertion for $Z$.

Let $W_{T,N}$ count pairs $(A,z)$ with $A\in\mathcal S_N(T)$,
$z\notin A$ adjacent to $A$, and $A\cup\{z\}$ unvisited.  If the union
avoids $\mathcal B_{N,2k+2}$, it induces a $(k+1)$-vertex tree.
There are $O_{d,k}(N)$ such pairs.  The other pairs number at most
$N^{o(1)}$, and their total expected contribution is $o(1)$ by
Proposition~\ref{prop:rough-capacity-bound}.  Therefore
\begin{equation}
 0\le Z_{T,N}-C_{T,N}\le W_{T,N},\qquad
 \bbE_{\pi_G}^G W_{T,N}
 \le N^{1-\gamma q_{d,k+1}+o(1)}+o(1).
 \label{eq:copy-component-comparison}
\end{equation}
Since $q_{d,k+1}-q_{d,k}=(d-2)/d>0$, this is $o(m_{T,N})$ below
$\gamma_{d,k}$ and $o(1)$ at $\gamma_{d,k}$.  Markov's inequality gives
the asserted convergence in probability for $C$ below the threshold.

At $\gamma=\gamma_{d,k}$, fix nonnegative integers $r_T$ with
$r=\sum_T r_T\ge1$ and expand the mixed factorial moment with $r_T$
copies of each shape.  Tuples avoiding
$\mathcal B_{N,2kr}$ and with mutually $2L_N$-separated copies number
$(1+o(1))N^r\prod_T\kappa_{d,T}^{r_T}$.  Their capacity is
$rq_{d,k}\theta_d^{-1}+o(1/\log N)$, giving limit $\prod_T\lambda_T(s)^{r_T}$.
For all other tuples, decompose the vertex union into $2L_N$-clusters.
Every selected copy is connected and therefore lies in a single cluster.
Among clusters disjoint from $\mathcal B_{N,2kr}$, let $a_1$ count those
containing exactly one selected copy and $a_2$ those containing at least
two.  Let $b$ count clusters meeting $\mathcal B_{N,2kr}$.

For fixed $a_1,a_2,b$, the number of tuples is at most
$N^{a_1+a_2+o(1)}$.  Indeed, choose one vertex in each cluster, using
$O(N)$ choices for a cluster disjoint from $\mathcal B_{N,2kr}$ and
$N^{o(1)}$ choices from $\mathcal B_{N,2kr}$ for a cluster meeting it.
All remaining vertices of that cluster lie in an $O_{k,r}(L_N)$-ball,
which has $N^{o(1)}$ vertices.  The finitely many assignments of the
$kr$ coordinate positions to clusters affect only a constant factor.
Distinct selected copies have distinct
vertex sets, including across different shapes.  Hence each cluster
counted by $a_2$ has capacity at least
$(q_{d,k}+\varepsilon_d)\theta_d^{-1}+o(1/\log N)$; those counted by $a_1$ have
capacity $q_{d,k}\theta_d^{-1}+o(1/\log N)$, and those counted by $b$ have capacity
at least $\delta_{d,kr}$.  Block additivity and the finite-target theorem
bound the contribution of these tuples by
\begin{equation}
 N^{-a_2\varepsilon_d/q_{d,k}
       -b\theta_d\delta_{d,kr}/q_{d,k}+o(1)}=o(1),
 \label{eq:late-tree-bad-mixed-factorial}
\end{equation}
because $a_2+b\ge1$.  Only finitely many triples $(a_1,a_2,b)$ occur at
each moment order.  The mixed factorial moments give the independent
Poisson limits for $Z$.  Equation~\eqref{eq:copy-component-comparison}
shows that the vectors $Z$ and $C$ coincide with probability $1-o(1)$.

For $\gamma>\gamma_{d,k}$, copies in $\mathcal S_N^0(T)$ have expected
count $N^{1-\gamma q_{d,k}+o(1)}=o(1)$.  All other copies contribute at
most $N^{o(1)}N^{-\gamma\theta_d\delta_{d,k}+o(1)}=o(1)$.
Markov's inequality proves their absence with probability $1-o(1)$.
Lemma~\ref{lem:late-set-start-transfer} gives all probability and
distributional conclusions uniformly over starting vertices.
\end{proof}

\begin{proof}[Proof of Corollary~\ref{cor:late-independence-transition}]
For the two-vertex tree $T$, $\kappa_{d,T}=d/2$ and
$D_N^{(\gamma)}(s)=Z_{T,N}^{(\gamma)}(s)$, so the count assertions follow
from Theorem~\ref{thm:late-component-hierarchy}.
Under stationary start the rough singleton bound excludes all unvisited
vertices of $\mathcal B_{N,3}$ with probability $1-o(1)$.
There are at most $N^{1+o(1)}$ pairs at distance at most $2L_N$.
For pairs disjoint from $\mathcal B_{N,3}$, local comparison gives
capacity at least $\frac{2(d-1)}{d}\theta_d^{-1}+o(1/\log N)$.  Thus the expected number of such
unvisited pairs is at most $N^{1-2\gamma(d-1)/d+o(1)}=o(1)$ when
$\gamma>\gamma_{d,2}$.  At $\gamma=\gamma_{d,2}$, pairs at distance
between $2$ and $2L_N$ and disjoint from $\mathcal B_{N,3}$ have capacity
at least
\[
 \widehat q_d\theta_d^{-1}+o(1/\log N),\qquad
 \widehat q_d:=\frac{2}{1+(d-1)^{-2}}>\frac{2(d-1)}{d},
\]
by Lemma~\ref{lem:tree-basic-capacities}.  Their expected unvisited count
is at most $N^{1-d\widehat q_d/(2(d-1))+o(1)}=o(1)$.
Consider the graph on $\mathcal U_N^{(\gamma_{d,2})}(s)$ that joins
distinct vertices at distance at most $2L_N$.  Consequently every edge
of this graph is an edge of $G_N$.  A connected component of this graph
with at least three vertices would contain a length-two graph path.
The endpoints would be a nonadjacent unvisited pair at distance two,
because the path lies in a tree neighbourhood.  This event has just been
excluded.  Thus each $2L_N$-cluster has one vertex or consists of the
endpoints of an original edge.  Lemma~\ref{lem:late-set-start-transfer}
gives uniformity over starts.  Finally, the independence event is
$\{D_N^{(\gamma)}(s)=0\}$, whose probability follows from the count limits.
\end{proof}

\begin{proof}[Proof of Corollary~\ref{cor:late-largest-component}]
The rough singleton bound and Lemma~\ref{lem:late-set-start-transfer}
exclude unvisited vertices in $\mathcal B_{N,k+1}$ with probability
$1-o(1)$, uniformly over starts.  On this event every component of size
at most $k+1$ is a tree, and the first $k+1$ vertices of a breadth-first
exploration of a larger component induce a tree as well.  The absence of
all finitely many shapes of size $k+1$ therefore excludes every larger
component when $\gamma>\gamma_{d,k+1}$, including cyclic components.
For $\gamma<\gamma_{d,k}$, Theorem~\ref{thm:late-component-hierarchy}
supplies a $k$-vertex component with probability $1-o(1)$.
At $\gamma=\gamma_{d,k}$, summing the independent limits over shapes
gives the stated Poisson law for the number of $k$-vertex components.
For $k\ge2$, the same theorem supplies a $(k-1)$-vertex component with
probability $1-o(1)$, because $\gamma_{d,k}<\gamma_{d,k-1}$.
The Poisson probability of zero then gives
\eqref{eq:critical-largest-component-law}.  For $k=1$, it gives the
stated probabilities of maximum component size zero and one.
\end{proof}

\end{document}